\documentclass[onefignum,onetabnum]{siamart190516}
\usepackage{amsmath,bm}                
\usepackage{graphicx}
\usepackage{mathrsfs}
\usepackage{amssymb}
\usepackage{epstopdf}
\usepackage{tikz-cd}
\usepackage{changes}
\usepackage{enumerate}
\usepackage{subfig}  
\usepackage{MnSymbol}
\DeclareMathOperator{\sech}{sech}

\allowdisplaybreaks

\newtheorem{Lemma}{Lemma}[section]
\newtheorem{Theorem}{Theorem}
\newtheorem{Proposition}[Lemma]{Proposition}
\newtheorem{Corollary}[Lemma]{Corollary}
\newtheorem{Remark}[Lemma]{Remark}

\newtheorem{Hypothesis}[Lemma]{Hypothesis}

\newcommand{\R}{\mathbb{R}}

\newcommand{\N}{\mathbb{N}}
\newcommand{\Z}{\mathbb{Z}}

\def\Re{\mathop{\mathrm{Re}}}

\def\coker{\mathop{\mathrm{Coker}\,}}
\newcommand{\rmO}{\mathrm{O}}

\newcommand{\rme}{\mathrm{e}}
\newcommand{\rmi}{\mathrm{i}}
\renewcommand{\ker}{\mathrm{Ker}\,}

\title{Numerical Methods for a Diffusive Class Nonlocal Operators
 \thanks{This work is supported in part by the National Science
  Foundation under grant DMS-1911742 (GJ).}}

\author{Loic Cappanera\thanks{Department of Mathematics, 
University of Houston, Houston, TX
(\email{lmcappan@central.uh.edu, gabriela@math.uh.edu, cward7@central.uh.edu}).}
\and Gabriela Jaramillo\footnotemark[1]
\and Cory Ward\footnotemark[1]}

\begin{document}
\emergencystretch 3em

\maketitle

\begin{abstract}
In this paper we develop a numerical scheme based on quadratures to approximate solutions
of integro-differential equations involving convolution kernels, $\nu$, of diffusive type.
In particular, we assume $\nu$ is symmetric and exponentially decaying at infinity.
We consider problems posed in bounded domains and in $\R$. In the case of bounded
domains with nonlocal Dirichlet boundary conditions, we show the convergence of 
the scheme for kernels that have positive tails, but that can take on negative values.
When the equations are posed on all of $\R$, we show that our scheme converges for 
nonnegative kernels. Since nonlocal Neumann boundary conditions lead to an equivalent
formulation as in the unbounded case, we show that these last results also apply to
the Neumann problem.
\end{abstract}



\section{Introduction}
In this paper we are interested in developing numerical algorithms for
approximating nonlocal evolution processes of the form
\begin{equation}\label{e:evolution}
u_t(x,t) = \int_\R ( u(y,t) - u(x,t) ) \nu(x,y) \;dy + f(x), \quad x\in \Omega \subset \R,
\end{equation}
where $\nu(x,y) = \nu(|x-y|)$ is a symmetric, extended (no compact support), 
and exponentially decaying function of diffusive type, 
which is not necessarily positive.

Integro-differential equations like the one above appear, for example, as model 
equations for diffusion processes that occur over fast time scales. 
They can be derived by first considering a reaction diffusion system, and then
using a Green's function to approximate the fast variable in terms of the slow variables.
The result is a nonlocal equation (or system) involving a convolution kernel.
This kernel describes a dispersion processes
that lies somewhere between regular 
diffusion, as expressed by the Laplace operator, and anomalous diffusion,
 modeled for instance using the Fractional Laplacian. 
More precisely, when viewed as probability density functions
for a random walk, the kernels considered here have finite second moment. 
Generally, this implies that the process has
a characteristic length scale, and one might be tempted to switch the convolution kernel for
the Laplace operator with an appropriate diffusion constant.
 However, as has been shown (see for example 
  \cite{shima2004, scholl2014, volpert2006, volpert2017, sherratt2018}) this simplification
misses the true character of the fast diffusion process and precludes one from finding 
interesting behavior, like for example chimera states \cite{shima2004}.

Other examples of systems that can be described using equation \eqref{e:evolution} 
come from population dynamics
\cite{mollison1977, kot1996, carillo2005},
and oscillating chemical reactions
 \cite{tanaka2003, shima2004, garcia2008}.
 Variations of the above integro-differential
equation also appear in other physical systems where nonlocal effects
are important.
 For instance, the references \cite{alberti1998, masi1994, masi1995, bates1997, Bates2008} 
explore a nonlocal continuum model for phase transitions, and in
\cite{bodnar2006} a model for the evolution of a particle system is
presented.
Peridynamic models also give rise to integro-differential equations like the
ones studied here, although in this case the convolution kernel is often assumed 
to have compact support.
 These models have been the subject of large study and 
are one of the driving forces behind the development of numerical schemes
for nonlocal models, 
\cite{silling2000, silling2003, silling2008, askari2008, weckner2009, 
silling2010, madenci2014, bobaru2016, chen2011, du2010, tian2013}.
Finally, extensions of the above equation in which the integral operator is
nonlinear, are also typical of neural field models.
In this case, the linearization about the homogenous steady state has the form of 
equation ~\eqref{e:evolution},
see for example \cite{amari1977, ermentrout1979, bressloff2011} and \cite{coombes2014}.
Of course, the literature presented here is not exhaustive and is just meant to give 
a general idea of the breath of applications that use nonlocal models.

Although in most applications the set $\Omega$ represents a physical domain
that is bounded, when the phenomenon of interest occurs at small
spatial scales compared to the size of this domain, it is reasonable to
pose the equation on all of $\R$. This is the case for example when
showing existence of traveling waves in predator prey models
\cite{dagbovie2014}, or existence of target patterns solutions in
oscillating chemical reactions \cite{jaramillo2018}. 
On the other hand, when both scales are comparable 
and $\Omega$ is considered to be
a bounded subset of $\R$, boundary conditions need to be formulated
carefully. Given that the model is now described by an
integro-differential equation, it is not enough to prescribe the value
of the solution or its derivatives at the boundary. Instead, boundary
conditions take the form of volume constraints, see
\cite{gunzburger2012, gunzburger2013}. Indeed, in various applications
volume constraints can provide an equivalent notion to Dirichlet and
Neumann boundary conditions that, moreover, is consistent with the
assumptions made in deriving evolution equations of the form \eqref{e:evolution}.
See also the discussion in  Appendix \ref{s:calculus}.

In both cases, bounded and unbounded $\Omega$, one is interested in
validating and guiding the mathematical analysis using numerical
simulations. Currently, one approach to approximate equation
\eqref{e:evolution} is to use an exponential time
difference scheme, \cite{kassam2005}. This consists in
picking a large domain, applying the Fourier transform to the spatial
variable, and then using an RK4 method to advance the time steps while
computing any nonlinearities in real space. 
One of the disadvantages of this approach is that the implied
periodic boundary conditions are not always desired. 
Moreover, computations can become costly if one also
needs a small spatial discretization to resolve small scale
phenomena. 
Alternatively, if the kernel $\nu(x,y)$ has as its Fourier
symbol a fractional polynomial, then it is possible to precondition
the equation with an appropriate differential operator and obtain as a
result a PDE, see for example \cite{laing2005}.
 One can then proceed to solve the problem using 
 finite differences, and impose local Dirichlet or Neumann
boundary conditions. The main draw back from this approach is that
the type of convolution kernels one can consider is restricted.

The goal of this paper is to propose a numerical method based on
quadratures for computing steady states,
\begin{equation}\label{e:steady_state}
\int_\R ( u(x,t) - u(y,t) ) \nu(|x-y|) \;dy = f(x), \quad x \in \Omega
\subset \R.
\end{equation}
which, in contrast to the methods mentioned above, accounts for nonlocal boundary
conditions. In particular, we provide schemes for approximating
solutions to \eqref{e:steady_state} when
\begin{enumerate}[i)]
\item $\Omega \subset \R$ is bounded and we know the value of the
  solution in $\Omega^c$.
\item $\Omega = \R$ and we assume the algebraic decay of the solution.
\item $\Omega \subset \R$ is bounded and we know the algebraic decay of the solution and the nonlocal flux
 from $\Omega^c$ into $\Omega$.
\end{enumerate}

Our scheme is adapted from \cite{oberman2014}, where the authors look
at item i)  in the particular case when $\nu(|x-y|)$ is the integral kernel associated with
the Fractional Laplacian, $(-\Delta)^{\alpha/2}$, $0< \alpha< 2.$ 
Our main contribution is to extend this scheme 
and provide a proof of convergence for all three problems (i)-(ii)-(iii) for
a larger range of kernels, meaning kernels that satisfy Hypotheses 
\ref{h:analyticity} and \ref{h:multiplicity}
and that are therefore exponentially decaying and do not have compact support.
Moreover, in the case of problems with nonlocal Dirichlet boundary conditions, kernels are
allowed to take negative values.

Notice that the properties exhibited by our kernels are in direct contrast 
to those considered in most of the literature pertaining to the numerical 
approximation of nonlocal equations. Indeed, most numerical schemes
 deal with either the integral form of the factional Laplacian,
\cite{oberman2014, acosta2017, acosta2017b, duo2018, duo2019} 
or with nonlocal operators involving kernels that are positive and compactly supported
\cite{delia2013, tian2013, gunzburger2013, tao2017, delia2020}.
In addition, since the problems considered here are posed on the
whole real line, there are additional difficulties  not 
encountered when looking at bounded domains with, for example,  
Dirichlet boundary conditions, or at problems that involve positive kernels with compact support.
Mainly the issue to be addressed is how to approximate the solution outside the computational domain. 
On the other hand, from a theoretical point of view
 it is not immediately clear that solutions exists and are unique 
 when problem  ~\eqref{e:steady_state} is
 posed on all of $\R$. 
 The second main contribution of this paper is to
adapt previous results from \cite{jaramillo2018} to show 
that the assumption of algebraic decay for solutions to 
equation ~\eqref{e:steady_state} 
leads to a well posed problem,
 provided that the right hand side, $f$, also has sufficient decay and satisfies 
 some compatibility conditions (zero mean and zero first moment).

Our results are organized as follows. In Section \ref{s:problem_eq}, we present the 
different nonlocal diffusion problems and boundary conditions that are considered in this 
paper. Details on the derivation of the equation \eqref{e:evolution}, as a model for 
population dynamic, and how nonlocal Dirichlet and Neumann boundary condition can be 
naturally defined are provided in Appendix \ref{s:calculus}.
In Section \ref{s:diffusive_operators}, we prove that the
nonlocal problem \eqref{e:steady_state} is well posed if the
equation is defined on a particular class of weighted Sobolev spaces.
In particular, we derive conditions on the right hand side, $f$, that guarantee
existence of a unique solution. 
These results hold for a large class
 of kernels, and for problems defined either on the whole real line
 or on bounded domains with nonlocal Neumann boundary condition. 
In this section we also state conditions that guarantee the problem is well posed
when considering nonlocal Dirichlet boundary conditions.  
In Section \ref{s:NumReal} we adapt the
methods from \cite{oberman2014}
to the problems i), ii) and iii).
The convergence of the numerical schemes is
established in Section \ref{s:cvg_proof}.
Unlike most schemes proposed in literature, the convergence of the Dirichlet problem
is established for kernels that can take negative values assuming the kernel has a positive tail.
 Finally, in Sections \ref{s:Dirichlet}, 
\ref{s:real_line}, and \ref{s:Neumann} we provide examples for cases
i), ii), and iii), respectively.

\section{Nonlocal diffusion model and boundary conditions}
\label{s:problem_eq}

As mentioned in the introduction,
the derivation of equations \eqref{e:evolution} and \eqref{e:steady_state} 
has been done in various contexts, see again for example
\cite{andreu2010,gunzburger2012, gunzburger2013,fife2003}. 
In the case of diffusion problems,
the key idea is to extend the concept of flux across a boundary to a version of flux that includes short 
as well as long range movement of particles. This extension is explained in detail in references
\cite{gunzburger2012, gunzburger2013}. For completeness,  in
Appendix \ref{s:calculus} we summarize some of the results of the above references, and use 
 an example from populations dynamics to derive equation \eqref{e:evolution}.

In this paper, we concentrate on the 1-dimensional steady state problem with 
nonlocal diffusion operator $\mathcal{L}$ of the form:
\[\mathcal{L} \ast u  =  \int_{\R }   (u(x,t) -u(y,t))  \nu(x,y)) \;dy, \]
where the kernel $\nu$ is assumed to be symmetric, meaning that $\nu(x,y)=\nu(|x-y|)$,
and exponentially decaying. We study three types of problems that are either defined on a bounded domain $\Omega$ with nonlocal Dirichlet or Neumann boundary conditions, or defined on the whole real line $\R$. These problems can be described as follows:

\begin{itemize}
\item Dirichlet boundary conditions 
\begin{equation}\label{e:Dirichlet}
\begin{array}{c c c}
\mathcal{L} \ast u = f &\mbox{for }& x \in \Omega,
\\ 
u = g &\mbox{for}& x \in \Omega^c,
\end{array} \tag{DP}
\end{equation}
\item Neumann boundary conditions
\begin{equation}\label{e:Neumann}
\begin{array}{c c c}
\mathcal{L} \ast u = f &\mbox{for }& x \in \Omega,
\\
 \mathcal{L}\ast u = f_c &\mbox{for }& x \in \Omega^c,
\end{array} \tag{NP}
\end{equation}
\item problem defined on whole real line
\begin{equation}\label{e:Unbounded}
\begin{array}{c c c}
\mathcal{L} \ast u = f &\mbox{for }& x \in \R.
\end{array} \tag{RP}
\end{equation}
\end{itemize}

The main goal of this paper is to show that the above problems are well posed, and to introduce numerical methods that approximate the solutions to these problems. We note that non stationary problems can also be solved with the method presented in section
\ref{s:num_method} by incorporating a time stepping scheme such as an RK-4 method.

\begin{Remark}\label{r:neumann}
Notice that in the case when the flux is prescribed, the equations for
the steady state takes the form
\[\mathcal{L} \ast u  = \bar{f}(x) \quad x \in \R,\]
 where
\[ \bar{f}(x)  = \left \{ \begin{array}{c c c}
 f(x) & \mbox{for} & x \in \Omega,
 \\ 
 f_c(x) & \mbox{for} & x \in \Omega^c.
\end{array} \right.
\]
\end{Remark}

\section{Weighted Spaces and Well-Posedness}
\label{s:diffusive_operators}

In this section we recall the results from \cite{jaramillo2019} where
it is shown that under certain assumptions on the kernel $\nu(x,y)$,
operators defined by equation \eqref{e:steady_state} are Fredholm
operators.  These results rely on a special class of weighted Sobolev
spaces, which we recall first before stating the assumptions on
$\nu(x,y)$ and $f(x)$. 
Our goal for this section is
to show that the equation,
\begin{equation}\label{e:main}
\mathcal{L} \ast u = f(x) \quad x\in \R,
\end{equation}
is well posed.

\subsection{Notation and Weighted Sobolev Spaces}\label{s:weighted_spaces}

For $s \in \N \cup\{0\}$, $p \in (1,\infty)$, and $\gamma \in \R$, we
let $M^{s,p}_\gamma(\R)$ denote the space of locally summable, $s$
times weakly differentiable functions $u:\R \rightarrow \R$ endowed
with the norm
\[ \| u \|_{M^{s,p}_\gamma(\R)} = \sum_{j =0}^s \| \partial_x^j u \|_{L^p_{j+\gamma}(\R)}
\quad \text{where} \quad
\| u \|_{L^{p}_\gamma(\R)} = \| ( 1+ | x|^2)^{\gamma/2} u \|_{L^p(\R)}.
\]
It is clear that for values of $\gamma>0$ these spaces impose a
certain level of algebraic decay, whereas for values of $\gamma <0$
functions are allowed to grow algebraically. This definition also
allows for the following embeddings: $M^{s,p}_\gamma(\R) \subset M^{s,p}_\sigma(\R)$ 
provided $\gamma > \sigma$, and $M^{s,p}_\gamma(\R) \subset M^{k,p}_\gamma(\R)$ if $s>k$.

Notice that we can extend the above definition to non integer values
of $s$ by interpolation and to negative values of $s$ by duality.
For values of $p \in (1, \infty)$ these spaces are also reflexive, so
that $(M^{s,p}_\gamma(\R) )^* = M^{-s,q}_{-\gamma}(\R)$, where $p$
and $q$ are conjugate exponents. The pairing between $f \in L^p_{\gamma}(\R)$ 
and an element in the dual space $ g \in L^q_{-\gamma}(\R)$ is given by the usual integral
\[ \langle f, g\rangle =  \int_\R fg \;dx.\]
In addition, if $p=2$ then the spaces $M^{s,2}_\gamma(\R)$ are Hilbert
spaces with inner product
\[ (f,g) = \sum_{j =0}^s \int_\R \partial_x^j f  \partial_x^j g( 1+ | x|^2)^{(j +\gamma)}\;dx.\]
The following lemma describes the algebraic decay of functions
belonging to $M^{1,p}_\gamma$ for positive values of the parameter
$\gamma$.
\begin{Lemma}\label{l:decay}
Given $\gamma >0$, a function $f \in M^{1,p}_\gamma(\R)$ satisfies
$|f(x)| \leq \|f'\|_{L^p_{\gamma+1}} |x|^{1/q- (\gamma +1)}$ as $|x| \rightarrow \infty$.
\end{Lemma}
\begin{proof}
Since $\gamma>0$ we may write 
$|f(x)| \leq \int_\infty^x |f'(y) (1+y^2)^{(\gamma +1)/2}| (1+y^2)^{-(\gamma+1)/2} \;dy$. 
The result then follows from H\"older's inequality.
\end{proof}

{\bf Notation: } In this paper we will also use the symbol $W^{s,p}_\gamma(\R)$ to 
denote the space of locally summable, $s$ times weakly differentiable functions that 
are bounded under the norm
\[ \| u \|_{W^{s,p}_\gamma(\R)} = \sum_{j =0}^s \|( 1+ | x|^2)^{\gamma/2} \partial_x^j u  \|_{L^p(\R)}.\]
In the case when $p=2$ we will also write $H^s_\gamma(\R) = W^{s,2}_\gamma(\R)$.
Furthermore, we will use $\langle , \rangle$ to denote
the pairing between an element in $M^{k,p}_\gamma(\R)$ and its dual
$M^{-k,q}_{-\gamma}(\R)$, and $(,)$ to denote the inner product on the
Hilbert spaces $M^{k,2}_\gamma(\R)$.

\subsection{Nonlocal Diffusive Operators on the Real Line} 
In this section we let $L(k)$ denote the Fourier symbol of the
operator $\mathcal{L}$. Our main assumptions are
\begin{Hypothesis}\label{h:analyticity}
The domain of the multiplication operator, $L(k)$, 
can be extended to a strip in the complex plane,
$\Omega = \R \times (-\rmi k_0, \rmi k_0)$ for some sufficiently small and
positive $k_0 \in \R$, and on this domain the operator is uniformly
bounded and analytic. Moreover, there is a constant $k_m \in \R$
such that the operator $L(k)$ is invertible with uniform bounds for
$|\Re k| > k_m$.
\end{Hypothesis}

Note that because $L(k)$ is analytic its zeros are isolated.
We can therefore assume that:
 
\begin{Hypothesis}\label{h:multiplicity}
The multiplication operator $L(k)$ has a zero, $k^*$, of multiplicity
$m$ which we assume is at the origin. Therefore, the symbol $L(k)$
admits the following Taylor expansion near the origin.
 \[ L(k ) = \alpha (- \rmi k)^m + \rmO(k^{m+1}), \quad \mbox{for} \quad k \sim 0 \quad \alpha = \pm 1.\]
\end{Hypothesis}

\begin{Remark}
In this paper we will consider the particular case when $m =2$, so
that this last assumption specifies that the operator behaves very
much like the Laplacian for small wavenumbers, giving its diffusive
character.
\end{Remark}

\begin{Remark}
Given that $\nu(x,y) = \delta(x-y) - \mathcal{L}$, the analyticity of
the symbol $L(k)$ implies that $\nu(x,y)$ is exponentially
localized. Similarly, because $L(k)$ has a zero of multiplicity
$m$ at the origin, then the first $m-1$ moments of the kernel
$\delta(x-y) - \nu(x,y) $ must be zero, while the $m$-th moment must be
bounded.
\end{Remark}

As was shown in \cite{jaramillo2019}, under the above hypotheses the
convolution operator $\mathcal{L}$ is a Fredholm operator in an
appropriate weighted space. This means in particular that the operator
has a closed range and a finite dimensional kernel and cokernel.
Here we define the cokernel of an operator as the kernel of its adjoint.

 The results presented in \cite{jaramillo2019} apply to more general
operators defined over $L^2(\R, Y)$, where $Y$ is a separable Hilbert
space, and that commute with the action of translations on
$L^2(\R,Y)$. Here we consider the case when $Y = \R$ and summarize the
results from \cite{jaramillo2019} in this next theorem.

\begin{Theorem}\label{p:fredholmconv}
Let $p \in (1,\infty)$ with $q$ its conjugate exponent, and let $\gamma \in \R$ be such that 
$\gamma+m+1/p \notin \{1,\cdots,m\}$. Suppose as well that the convolution
operator $\mathcal{L}: M^{m,p}_\gamma(\R) \rightarrow W^{l,p}_{\gamma+m}(\R)$ 
satisfies Hypotheses \ref{h:analyticity} and \ref{h:multiplicity}. Then, with appropriate 
value of the integer $l$, the operator is Fredholm and
\begin{itemize}
\item for $\gamma <1-m - 1/p$ it is surjective with kernel spanned by $\mathbb{P}_m$;
\item for $\gamma > - 1 + 1/q$ it is injective with cokernel spanned by $\mathbb{P}_m$;
\item for $j -1 -m +1/q< \gamma <j +1-m - 1/p$ , where $j \in \mathbb{N}$, $1\leq j < m$,
 its kernel is spanned by $\mathbb{P}_{m-j}$ and its cokernel is spanned by $\mathbb{P}_j$.
\end{itemize}
Here $\mathbb{P}_m$ is the $m$-dimensional space of all
polynomials with degree less than $m$.
\end{Theorem}
The above results follows from Lemma \ref{l:decompositionconv} and
Proposition \ref{p:fredholm2}, which 
show that under the above hypotheses the convolution operators
considered here can be written as the composition of an invertible
operator and a Fredholm operator. These results can also be found in \cite{jaramillo2019}.

\begin{Lemma}\label{l:decompositionconv}
Let the multiplication operator $L(k)$ satisfy Hypothesis
\ref{h:analyticity}-\ref{h:multiplicity}. 
Then $L(k)$ admits the following decomposition:
\[ L(k) = M_L(k) L_{NF}(k) = L_{NF}(k) M_R(k),\]
where $L_{NF}(k) = (-\rmi k)^m / (1\pm \rmi k)^l$, while $M_{L/R}(\xi)$
and their inverses are analytic and uniformly bounded on $\Omega$.
\end{Lemma}
 
 \begin{Proposition}\label{p:fredholm2}
Let $m$ and $l$ be non negative integers, and $p \in (1, \infty)$ with $q$ its conjugate exponent.
Then, the operator
\[ (1\pm \partial_x)^{-l}\partial^m_x : M^{m,p}_{\gamma}(\R) \longrightarrow W^{l,p}_{\gamma+m}(\R)\]
is Fredholm for $\gamma + m+1/p \notin \{ 1, \cdots, m\}$. In particular,
\begin{itemize}
\item for $\gamma <1-m - 1/p$ it is surjective with kernel spanned by $\mathbb{P}_m$;
\item for $\gamma > - 1 + 1/q$ it is injective with cokernel spanned by $\mathbb{P}_m$;
\item for $j -1 -m +1/q< \gamma <j +1-m - 1/p$ , where $j \in \mathbb{N}$, $1\leq j < m$, 
its kernel is spanned by $\mathbb{P}_{m-j}$ and its cokernel is spanned by $\mathbb{P}_j$.
\end{itemize}
For $\gamma + m+ 1/p \in {1,\cdots,m}$ the operator does not have a
closed range.  Here $\mathbb{P}_m$ is the $m$-dimensional space of all
polynomials with degree less than $m$.
\end{Proposition}

Heuristically, the main reason why operators of the form 
$(1 \pm \partial_x)^{-\ell}\partial_x^m$ are not Fredholm in regular Sobolev
spaces is because they have a zero eigenvalue embedded in their
essential spectrum. In particular, this means that 
one can use the corresponding eigenfunction to construct Weyl sequences
and consequently show that the operator does not have closed range.

For example, consider the one dimensional Laplacian
 $\partial_x^2: H^2(\R) \longrightarrow L^2(\R)$. 
Its nullspace is spanned by $\{1,x\}$, and although these functions are not in
$ H^2(\R) $, one can use them to construct Weyl sequences. For example,
let $u_n = \chi( |x|/n) $, where  $\chi(|x|)$ is a smooth radial function
 equal to one when  $|x| <1$, and  equal to zero when $|x|>2$. 
 Notice that this sequence does not converge in $H^2(\R)$. 
 However $\| \partial_x^2 u_n \|_{L^2} \rightarrow 0$ as $n \rightarrow \infty$,  
 showing that the operator does not have a closed range.

The reason for considering $\partial_x^2 :M^{s,2}_\gamma(\R) \longrightarrow L^2_{\gamma+2}(\R) $ 
is that by picking large positive values of $\gamma$, and thus imposing algebraic decay,
one no longer has the result $\| \partial_x^2 u_n \|_{L^2_{\gamma+2}} \rightarrow 0$. 
On the other hand, by picking negative values of $\gamma$ and allowing algebraic growth,
the sequence $u_n = \chi( |x|/n) $ no longer converges to an element in the domain $M^{s,2}_\gamma(\R) $. 

In this paper we will restrict ourselves to functions $f(x)$ in
weighted Sobolev spaces, $L^p_{\gamma}(\R)$, that impose a high degree of
algebraic decay. As a result our convolution operators will have a two
dimensional cokernel spanned by at most $\{1, x\}$ (since we are
assuming $m=2$). 

The goal for us is to reformulate the problem so that we deal with an invertible
operator. This means that we will look at the following system
\begin{equation}\label{e:main1}
 \mathcal{L} \ast u + a_1 \mathcal{L} \ast P_1(x) + a_2 \mathcal{L} \ast P_2(x) = f(x) \quad x\in \R,
\end{equation}
where $f(x) \in L^p_\gamma(\R)$ is given, $u(x)$ and $a_i \in \R$,
with $ i \in \{1,2\} $, represent the variables we want to solve for,
and $\mathcal{L} \ast P_i(x) \in C^\infty(\R),$ are functions that
span the cokernel of our operator. In particular we require
\[ \langle \mathcal{L} \ast P_1, 1\rangle = \int_\R \mathcal{L} \ast P_1(x) \;dx =1
 \qquad \langle \mathcal{L} \ast P_1, x \rangle = \int_\R \mathcal{L} \ast  P_1(x) \cdot x \;dx =0 \]
\[ \langle \mathcal{L} \ast P_2, 1\rangle = \int_\R \mathcal{L} \ast P_1(x) \;dx =0 
\qquad \langle \mathcal{L} \ast P_2, x \rangle = \int_\R \mathcal{L} \ast P_2(x) \cdot x \;dx =1 \]

For example, one may pick $P_1(x) = (1/2) \log(\cosh (x))$ and $P_2(x) = \partial_x P_1(x) = (1/2)\tanh(x)$.

The above discussion leads to the following proposition.
\begin{Proposition}\label{p:L_invertible}
Given $\gamma > -1 + 1/p$, the convolution operator $\mathcal{L}$ with
Fourier symbol $L(k) = M_L(k) (ik)^2/ (1+ k^2)$ and defined as
\[\begin{array}{c c c}
 \mathscr{L}: M^{2,p}_{\gamma}(\R) \times \R\times \R & \longrightarrow & W^{2,p}_{\gamma+2}(\R) \\ 
 (u, a_1, a_2) & \mapsto & \mathcal{L} \ast ( u + a_1 P_1 + a_2 P_2)
 \end{array}
\]
is invertible, and therefore well defined.
\end{Proposition}

We also have the following corollary, which gives conditions on the
right hand side of equation \eqref{e:main} that guarantee existence of
solutions. This result is a consequence of the previous proposition
and Lemma \ref{l:decay}.
\begin{Corollary}\label{c:invertible_con}
Let $s \in \Z \cup [2, \infty)$, and $p \in (1, \infty)$ with $q$ its conjugate exponent.
 Consider the convolution operator $\mathcal{L}$, with Fourier symbol
$L(k) = M_L(k) (ik)^2/ (1+ k^2)$, and defined as
\[\begin{array}{c c c}
 \mathscr{L}: M^{s,p}_{\gamma}(\R) & \longrightarrow & W^{s,p}_{\gamma+2}(\R)\\ 
 u &\longmapsto & \mathcal{L} \ast u .
 \end{array}
\]
Suppose $\gamma > -1 + 1/q$, then the equation $\mathcal{L} \ast u =f$ 
has a unique solution, with $|u(x)| < C |x|^{1-1/p-(\gamma+1)}$ for
large $|x|$, provided the right hand side $f(x) \in W^{s,p}_{\gamma+2}(\R) $ satisfies
\[ \langle f, 1 \rangle = 0 \qquad \mbox{and} \qquad \langle f, x \rangle = 0.\]

\end{Corollary}

\begin{Remark}
Notice that by Lemma \ref{l:decay}, if the function $f \in
W^{2,p}_{\gamma+2}(\R)$, then for large $|x|$ we have that 
$ |f (x)| < \|f'\|_{L^p_{\gamma+2}} |x|^{1/q-( \gamma + 3)}$, 
where $p$ and $q$ are conjugate exponents.
\end{Remark}

\begin{Remark}
If in addition to Hypotheses \ref{h:analyticity} and \ref{h:multiplicity},
the kernel $\nu(x) \in L^2(\R)$, and the equation ~\eqref{e:main} is posed on
a bounded domain, then 
\[ \mathcal{L} \ast u = u(x) - \int_\R \nu(|x-y|) u(y) \;dy = u(x) - \int_\Omega \nu(|x-y|) u(y) \;dy, \]
defines an operator $\mathcal{L} : L^2(\Omega) \rightarrow L^2(\Omega)$ 
which is a compact perturbation of the identity. 
This follows since the integral in the above expression corresponds to a Hilbert-Schmidt operator.
As a result, problem (DP)  has a unique solution.
\end{Remark}

\begin{Remark}
The above results can be extended to operators $\mathcal{L}$ defined
on lattices. The discussion follows again from the results
presented in \cite{jaramillo2019}, it is summarizes in appendix
\ref{a:well_posedness_nonlocal_lattice} for completeness.
\end{Remark}

\section{A Numerical Method  for nonlocal diffusive operators}\label{s:NumReal}

In this section, we extend  the discretization scheme presented 
in Huang and Oberman's  paper \cite{oberman2014} so it is valid for problems
with more general kernels defined on the whole real line.
In particular, the method presented here applies to equations of the form
\begin{equation}\label{e:main2}
 \mathcal{L} \ast u (x) = \int_\R (u(x) - u(x-y) )\nu(y) \; dy = f(x)
 \quad x \in \R
 \end{equation}
where $\nu(y)$ is an exponentially localized kernel, so that the Fourier
symbol $L(k)$ satisfies Hypotheses \ref{h:analyticity} and
 \ref{h:multiplicity} with $m =2$.
In terms of the moments of $\nu(y)$, these assumptions lead to
\begin{align}\label{e:ass1}
\int_\R \nu(y) \;dy = 1 & \quad, \quad  \int_\R \nu(y) y \; dy = 0.
\end{align}
Additionally, in order to bound the local truncation error in the numerical schemes, we also make the following assumptions
\begin{align}\label{e:ass2}
 \int_\R \nu(y) y^2 \; dy < \infty & \quad, \quad  \int_\R \nu(y) y^4 \; dy < \infty \quad, \quad   u \in C^4(\R).
\end{align}

In Remark \ref{r:neumann}, we noted that equation \eqref{e:main2} also
encompasses problems posed on a bounded domain with nonlocal Neumann
boundary conditions. Thus, the results presented in this section also
apply to this type of situations.
Similarly, the result presented here also apply to Dirichlet boundary problems where the equation \eqref{e:main2} is considered on a bounded domain $\Omega \subset \R$.
 
\subsection{Discretization of the Operator}\label{ss:num_scheme}
As a first step, we set up a numerical grid defined by $x_i=ih$, 
$i \in \mathbb{Z}$ and $h>0$. We then split Eq.~\eqref{e:main2} into a
(possibly) singular part and a tail:
\begin{equation*} 
\mathcal{L} \ast u(x) = \int^h_{-h} [u(x) - u(x-y)] \nu(y) \; dy 
+ \int_{|y| \geq h}[u(x) - u(x-y)] \nu(y) \; dy.
\end{equation*}
 We denote the first and second integral by $\mathcal{L}_S \ast u(x)$ and $\mathcal{L}_T \ast u(x)$, respectively.

\subsubsection{Discretization of the Singular Integral}
We first rewrite the singular integral, considering it as a Cauchy
P.V.:
\begin{align*}
\mathcal{L}_S \ast u(x) &= \int^h_{-h} [u(x) - u(x-y)] \nu(y) \; dy \\[.2cm]
&:= \lim_{\epsilon \to 0} \int_\epsilon^h [u(x) - u(x-y)] \nu(y) \; dy +\int^{-\epsilon}_{-h} [u(x) - u(x-y)] \nu(y) \; dy \\[.2cm]
&= \int_0^h [2u(x) - u(x+y) - u(x-y)] \nu(y) \; dy.
\end{align*}
The last equality follows from changing variables, $z=-y$, in the
second integral and using the evenness of $\nu$.

Assuming $u \in C^4$, we can use Taylor's Theorem and expand $u(x-y),u(x+y)$ to obtain
that the above integral is 
\begin{equation*}
-u''(x) \int_0^h y^2 \nu(y) \; dy - \frac{u^{(4)}(\xi_1)}{12}\int_0^h y^4 \nu(y) \; dy,
\end{equation*}
where $\xi_1 \in (x-h,x+h)$ is chosen appropriately to ensure that the equality holds.
We can also rewrite $u''(x)$, using a Taylor expansion, to get its
second order finite difference formula:
\begin{equation*}
-\bigg[\frac{u(x+h) -2u(x) + u(x-h)}{h^2} 
+ \frac{u^{(4)}(\xi_2)}{12}h^2\bigg]\int_0^h y^2 \nu(y) \; dy 
- \frac{u^{(4)}(\xi_1)}{12}\int_0^h y^4 \nu(y) \; dy.
\end{equation*}
Simplifying this result, we have
\begin{equation*}
\mathcal{L}_S \ast u(x) = -\bigg[u(x+h) -2u(x) + u(x-h)\bigg] f_1(h)
-\frac{u^{(4)}(\xi_2)}{12}f_2(h) - \frac{u^{(4)}(\xi_1)}{12} f_3(h)
\end{equation*}
where
\begin{equation*}
f_1(h) = \frac{1}{h^2}\int^h_{0} y^2 \nu(y) \; dy, \quad \quad f_2(h)
=h^2 \int^h_{0} y^2 \nu(y) \; dy, \quad \quad f_3(h) = \int^h_{0} y^4 \nu(y) \; dy.
\end{equation*}
For a specific grid point $x_i$, we can rewrite the above formula as follows:
\begin{align*}
\mathcal{L}_S \ast u(x_i) =& \bigg[u(x_i) - u(x_{i-1})\bigg] f_1(h) 
+ \bigg[u(x_i) - u(x_{i+1})\bigg] f_1(h)-\frac{u^{(4)}(\xi_2)}{12}f_2(h) \\ 
&- \frac{u^{(4)}(\xi_1)}{12} f_3(h).
\end{align*}

\subsubsection{Discretization of the Tail Integral}
Let $T(x)$ be the hat function
\begin{equation*}
T(x) :=
\begin{cases}
1-\frac{|x|}{h} & \text{if } |x| \leq h, \\ 
0 & \text{otherwise}.
\end{cases}
\end{equation*}
Then we can interpolate any function $f(x)$ on all of $\mathbb{R}$ as follows:
\begin{equation*}
Pf(y) := \sum\limits^\infty_{j=-\infty} f(x_j)T(y-x_j).
\end{equation*}
Note that this is just piecewise polynomial interpolation, where we've
chosen the interpolating polynomials to be the linear (Lagrange)
polynomials on their given domain $[x_i,x_{i+1}]$. \\

Letting $f(y) = u(x_i) - u(x_i-y)$ and plugging its interpolation
into the tail integral we get
\begin{align*}
\mathcal{L}_T \ast u(x_i) =& \int_{|y| \geq h} f(y) \nu(y) \; dy\\[.2cm]
\approx & \int_{|y| \geq h} Pf(y) \nu(y) \; dy\\[.2cm]
= & \sum\limits^\infty_{j=-\infty} \bigg[(u(x_i)-u(x_i - x_j)) \int_{|y| \geq h} T(y-x_j)\nu(y) \; dy \bigg].
\end{align*}
Because the hat function $T$ is zero almost everywhere, the
latter integral is actually defined on finite interval and, as we will
show later, it can be computed easily. Finally, because this is a
polynomial interpolation, it can be shown that
\begin{equation*}
\mathcal{L}_T \ast u(x_i)  := \int_{|y| \geq h} Pf(y) \nu(y) \; dy  + O\bigg(h^2 \int^\infty_h |\nu(y)| \; dy\bigg).
\end{equation*}

\subsubsection{Discretization of the Nonlocal Operator} \label{ss:discrete_nonlocal}
Let
\begin{equation} 
\begin{split}
f_1(h) = \frac{1}{h^2} \int^h_0 y^2 \nu(y) \; dy, \quad  \quad
 f_3(h) = \int^h_0 y^4 \nu(y) \; dy, \\[.2cm]
f_2(h) = h^2 \int^h_0 y^2 \nu(y) \; dy, \quad \quad 
 f_4(h) = h^2 \int^\infty_h  | \nu(y)| \; dy,
\label{e:def_fi}
\end{split}
\end{equation}
then using the above results, we can write
\begin{align}  \label{e:order}
\mathcal{L} \ast u(x_i) =& \mathcal{L}_S \ast u(x_i) + \mathcal{L}_T \ast u(x_i) \nonumber \\[.2cm] 
=& \sum\limits^\infty_{j=-\infty}\bigg([u(x_i)-u(x_i - x_j)]w_j\bigg) +\\[.2cm]
&O\bigg(f_2(h)\bigg) + O\bigg(f_3(h)\bigg) + O\bigg(f_4(h)\bigg), \nonumber
\end{align}
where
\begin{equation}\label{e:weights1}
w_j =
\begin{cases}
0 & \text{if } j=0,\\[.2cm]
f_1(h) + \int_{|y| \geq h} T(y-x_j) \nu(y) \; dy & \text{if } j = \pm 1, \\[.2cm]
\int_{|y| \geq h} T(y-x_j) \nu(y) \; dy &
\text{otherwise.}
\end{cases}
\end{equation}
Note that whenever $j=0$, we have that $u(x_i)-u(x_i - x_j)=0$ and
hence we can define $w_0$ arbitrarily.

\begin{Remark}\label{r:order}
From equation \eqref{e:order} we can conclude that the order of the scheme presented in this section is given by
\[ \min \left\{O\bigg(f_2(h)\bigg) + O\bigg(f_3(h)\bigg) + O\bigg(f_4(h)\bigg) \right\}.\]

\end{Remark}

\subsection{Numerical Methods on a Finite Lattice }\label{s:num_method}
Having found a discretization of the nonlocal operator
that is also valid on the whole
real line, we now focus on to its practical application. Namely, while the
scheme approximates the nonlocal equations for any $x\in\Omega$, where $\Omega$ can be a bounded or unbounded subset of $\R$, it
still requires the calculation of an infinite number of weights,
$w_j$. In this section, we discuss how to modify the scheme such that
only a finite number of weights need to be calculated as well as a few
nontrivial integrals.

First, let $M$ be some even number such $L=\frac{M}{2}h$ and let
$L_W=2L=Mh$. Define $x_j = jh$ for $-M \leq j \leq M$. Now, unlike in
the previous section, we want to split the nonlocal operator as
\begin{align*}
\mathcal{L} \ast u(x_i) =& \int^h_{-h} [u(x_i) - u(x_i-y)] \nu(y) \; dy + \int_{h \leq |y| \leq L_W}[u(x_i) - u(x_i-y)] \nu(y) \; dy \\[.2cm]
& + u(x_i) \int_{|y| \geq L_W} \nu(y) \; dy - \int_{|y| \geq L_W}u(x_i -y) \nu(y) \; dy.
\end{align*}
Call these integrals (Ia), (Ib), (II), and (III) respectively. Note
that integral (III) depends on the specific point $x_i$ chosen.

Due to the local nature of the hat functions, we can repeat all of the
arguments of Section~\ref{ss:num_scheme} to write
\begin{align*}
\text{(Ia)} + \text{(Ib)} =& \sum\limits^M_{j=-M}\bigg([u(x_i)-u(x_i - x_j)]w_j\bigg)\\[.2cm]
&+O\bigg(f_2(h)\bigg) + O\bigg(f_3(h)\bigg) + O\bigg(f_4(h)\bigg)
\end{align*}
where
\begin{equation*}
w_j =
\begin{cases}
0 & \text{if } j=0,\\[.2cm]
f_1(h) + \int_{h \leq |y| \leq L_W} T(y-x_j) \nu(y) \; dy & \text{if } j = \pm 1, \\[.2cm]
\int_{h \leq |y| \leq L_W } T(y-x_j) \nu(y) \; dy & \text{if } 1 < |j| \leq M,
\end{cases}
\end{equation*}
with the functions $f_k(h)$ defined in \eqref{e:def_fi}.
Note also that the weights $w_j$ are still even here as well.

Integral (II) doesn't depend of $u(x)$ and can, in principle, be
calculated analytically. Thus, for simplicity we define the 
constant $A$ as
\begin{equation}\label{e:def_A}
A = \int_{|y| \geq L_W} \nu(y) \; dy .
\end{equation}
Lastly, integral (III) can also be calculated analytically depending
on the specific problem under consideration.
This step is described in the following.

\subsection{Dirichlet Problem} \label{ss:scheme_Dirichlet}
We begin here by considering the Dirichlet problem
\begin{equation}\label{e:DP}
\begin{cases}
\mathcal{L} \ast u(x) = f(x), & x \in (-L, L), \\ 
u(x) = g(x), & x \in (-L,L)^c.
\end{cases}
\end{equation}
Note that $L_W$ is the smallest number such that  $u(x_i - y) = g(x_i - y)$ 
for all $|y| \geq L_W$ and all $-\frac{M}{2}+1 \leq i \leq \frac{M}{2}-1$.
Hence, defining $B_i = (III)$ to
highlight it's dependence on $i$, we have
\begin{equation} \label{e:def_Bi}
B_i = \int_{|y| \geq L_W} g(x_i -y) \nu(y) \; dy.
\end{equation}
Since $\nu$ and $g$ are given functions, the integral $B_i$ can be
calculated analytically. Dropping the big oh terms, our numerical
scheme for the Dirichlet problem is given by
\begin{equation}\label{e:disDP}
\sum\limits^M_{j=-M}\bigg([u_i-u_{i -j}]w_j\bigg) + Au_i - B_i = f(x_i),
\end{equation}
for $-\frac{M}{2}+1 \leq i \leq \frac{M}{2}-1$. 
Note that $u_{\pm \frac{M}{2}} = g(x_{\pm \frac{M}{2}})$ 
and so the boundary points, $i=\pm \frac{M}{2}$, do not need to be solved for.

\subsection{Whole Real Line Problem} \label{ss:scheme_whole_real_line}
In this section we consider the problem posed on the whole real line
\begin{equation}\label{e:UP}
\mathcal{L} \ast u(x) = f(x), \quad x \in \mathbb{R}.
\end{equation}

Since $\nu(x)$ and $f(x)$ are given explicitly it may be possible in
some cases to discern the asymptotic decay rate of the solution
$u(x)$.  For example, if $u(x)$ decays exponentially to zero then it
should be possible to ignore the integral (III) by choosing $L$
sufficiently large. This approximation is essentially the Dirichlet
problem previously considered where we take $g(x) = 0$ (on a
sufficiently large domain).

On the other-hand, if the solution $u(x)$ decays too slowly 
(e.g. algebraically) to
zero then ignoring the integral (III) could lead to large errors (or
require extremely large domain sizes). To get around this issue, 
using corollaries \ref{c:invertible_con}-\ref{c:uniqueness_d},
we know the approximate asymptotic decay rate of $u(x)$ i.e. $u(x) \sim
g(x)$. Then to get a good approximation of the integral (III), we may
assume
\begin{equation}\label{e:decay_g}
u(x) \approx
\begin{cases}
u(L) \frac{g(x)}{g(L)}, \quad x \geq L, \\ 
u(-L) \frac{g(x)}{g(-L)},
\quad x \leq -L.
\end{cases}
\end{equation}
Essentially this just assumes the decay rate is a good approximation
of the solution outside the interval $[-L,L]$. 

The advantage of this formulation, compared to the Dirichlet problem, 
is that we
only assume the decay of the solution to be known. It allows us to use smaller 
values of $L$ to get a given order of accuracy compared to the Dirichlet problem 
which assumes that the solution $u(x$) vanished for large values of $x$. 
Thus, the Dirichlet method is 
either adequate for problem with exponentially decaying solution or 
 requires large  value of $L$ when the solution decays algebraically.

Thus the integral (III) can be approximated as
\begin{align*}
\text{(III)} &= \int_{|y| \geq L_W}u(x_i -y) \nu(y) \; dy \\[.2cm]
&\approx \frac{u(L)}{g(L)} \int_{|y| \geq L_W}g(x_i -y) \nu(y) \; dy \\[.2cm]
&\approx  \frac{u(-L)}{g(-L)} \int_{-\infty}^{-L_W}g(x_i -y) \nu(y) \; dy + \frac{u(L)}{g(L)} \int_{L_W}^\infty g(x_i -y) \nu(y) \; dy \\[.2cm]
&=  u(-L) B^1_i + u(L) B^2_i \\[.2cm]
&=  u_{-\frac{M}{2}} \; B^1_i + u_{\frac{M}{2}} \; B^2_i,
\end{align*}
where
\begin{equation} \label{e:def_B1i_B2i}
B^1_i = \int_{-\infty}^{-L_W} \frac{g(x_i -y)}{g(-L)} \nu(y) \; dy, \quad  
B^2_i =\int_{L_W}^\infty \frac{g(x_i -y)}{g(L)} \nu(y) \; dy.
\end{equation}
We now see that the integrals $B^1_i$ and $B^2_i$ can be calculated
analytically. This is slightly different from the Dirichlet problem we
considered before as $u(-L)$ and $u(L)$ are now unknowns and
must be solved for in the numerical scheme itself.  With that said, we
can write the complete scheme as
\begin{equation*}
\sum\limits^M_{j=-M}\bigg([u_i-u_{i - j}]w_j\bigg) + Au_i - u_{-\frac{M}{2}} \; B^1_i - u_{\frac{M}{2}} \; B^2_i = f(x_i)
\end{equation*}
for all $-\frac{M}{2}\leq i \leq \frac{M}{2}$ and it's understood that whenever 
$|i-j| > \frac{M}{2}$ we replace $u_{i-j}$ with the approximation (\ref{e:decay_g}).
Notice that the range of allowed $i$ values has increased to account for the fact 
that $u_{-\frac{M}{2}}$ and $u_{\frac{M}{2}}$ are now unknowns.

\subsection{Neumann Problem} \label{ss:scheme_Neumann}
Here we consider the Neumann problem
\begin{equation*}
\begin{cases}
\mathcal{L} \ast u(x) = f(x), \quad &x \in (-\tilde{L},\tilde{L}), \\ 
\mathcal{L} \ast u(x) = f_c(x), \quad &x \in (-\tilde{L},\tilde{L})^c,
\end{cases}
\end{equation*}
where the solution is assumed to decay algebraically, meaning that the equation \eqref{e:decay_g} is satisfied for a given $g$ and $L>\tilde{L}>0$.
Recall that this is equivalent to solving the whole real line problem
\begin{equation*}
\mathcal{L} \ast u(x) = \bar{f}(x), \quad x \in \mathbb{R} \\
\end{equation*}
where
\begin{equation*}
\bar{f}(x) =
\begin{cases}
f(x), \quad &x \in (-\tilde{L},\tilde{L}), \\
 f_c(x), \quad &x \in (-\tilde{L},\tilde{L})^c.
\end{cases}
\end{equation*}
Hence, since we've already developed a numerical method which solves
the problem on the whole real line, we can apply it verbatim to also
solve Neumann problems.

\section{Proofs of Convergence}\label{s:cvg_proof}

In this section we established the convergence of the numerical schemes introduced in the previous section, meaning the schemes approximating the solution of problems of the type \eqref{e:Dirichlet}, \eqref{e:Unbounded} and \eqref{e:Neumann}.

\subsection{Dirichlet}
Consider again the Dirichlet scheme
\begin{equation*}
\sum\limits^M_{j=-M}\bigg([u_i-u_{i -j}]w_j\bigg) + Au_i - B_i = f(x_i)
\end{equation*}
for $-\frac{M}{2}+1 \leq i \leq \frac{M}{2}-1$. 
In this section we will write this system in matrix form and then derive bounds on 
the eigenvalues of the corresponding matrix. Together with the local truncation error derived earlier, this will show the scheme converges.

In addition to assumptions (\ref{e:ass1}) and (\ref{e:ass2}), we will also assume in this section that $\nu(y)$ is an arbitrary $L^1(\R)$ function, taking positive or negative values, such that for all sufficiently large values of $L$ we have that the tails are strictly positive:
\begin{equation}\label{e:ass3}
\int_{|y| >2 L} \nu(y) \; dy > 0.
\end{equation}
This hypothesis will be sufficient to show that the scheme is stable.

\subsubsection{Stability}
To write the scheme in matrix form, we focus first on the summation and, for ease of presentation, we let $u_k:=u(x_k)$. We then have
\begin{align*}
\sum\limits^M_{j=-M}\bigg([u_i-u_{i-j}]w_j\bigg) &= u_i\sum\limits^M_{j=-M}w_j -\sum\limits^M_{j=-M} u_{i-j}w_j \\[.2cm]
&= u_i\sum\limits^M_{j=-M}w_j -\sum\limits_{|i-j| \leq \frac{M}{2}-1} u_{i-j}w_j -\sum\limits_{|i-j| > \frac{M}{2}-1} u_{i-j}w_j .
\end{align*}
Note that in the last sum of the second line we have $u_{i-j}=g_{i-j}$. 
Keeping in mind that $i$ is a fixed constant here, it reads
\begin{align*}
\sum\limits^M_{j=-M}\bigg([u_i-u_{i-j}]w_j\bigg)  &= u_i\sum\limits^M_{j=-M}w_j 
-\sum\limits_{j= i- \frac{M}{2}+1}^{j= i+ \frac{M}{2}-1} u_{i-j}w_j 
-\sum\limits^{i- \frac{M}{2}}_{j=-M} g_{i-j}w_j 
-\sum\limits_{j=i+ \frac{M}{2}}^{M} g_{i-j}w_j .
\end{align*}
Renaming the above quantities as follows
\begin{equation*}
c^1=\sum\limits^M_{j=-M}w_j  \quad, \quad c^2_i=\sum\limits^{i- \frac{M}{2}}_{j=-M} g_{i-j}w_j \quad , \quad c^3_i=\sum\limits_{j=i+ \frac{M}{2}}^{M} g_{i-j}w_j,
\end{equation*} 
we get
\begin{align*}
\sum\limits^M_{j=-M}\bigg([u_i-u_{i-j}]w_j\bigg) &= u_i c^1 - c^2_i - c^3_i
-\sum\limits_{j= i- \frac{M}{2}+1}^{j= i+ \frac{M}{2}-1} u_{i-j}w_j \\[.2cm]
&= u_i c^1- c^2_i - c^3_i - \sum\limits_{k=-\frac{M}{2}+1}^{\frac{M}{2}-1} u_{k}w_{i-k}.
\end{align*}
The last equality is obtained by the change of variables $k=i-j$. Letting $\hat{u}$ denote
 the vector
\begin{equation*}
\begin{bmatrix}
u_{-\frac{M}{2}+1} \\
\vdots \\
u_{\frac{M}{2}-1}
\end{bmatrix},
\end{equation*}
we can rewrite the above equation in matrix form as
\begin{equation*}
c^1 \hat{u} - c^2  - c^3 - \hat{w}\hat{u},
\end{equation*}
where $c^2, c^3$ are  vectorized versions of $c^2_i, c^3_i$ and $\hat{w}$ is the
 matrix given by
\begin{equation*}
\begin{bmatrix}
w_0 & w_{-1} & \dots & w_{-M+2} \\
w_1 & w_0 & \dots & w_{-M+3}\\
\vdots & \vdots & \ddots\ & \vdots \\
w_{M-2} & w_{M-3} & \dots & w_0
\end{bmatrix}.
\end{equation*}
Note that since $w_{-j}=w_j$ the matrix is symmetric and in fact Toeplitz.
We can then write the numerical scheme for the Dirichlet problem as
\begin{equation*}
c^1 \hat{u} - c^2  - c^3 - \hat{w}\hat{u} + A\hat{u} - B = \hat{f}
\end{equation*}
or, equivalently,
\begin{equation*}
(c^1 I - \hat{w} + A I) \hat{u} = \hat{f} + c^2 +c^3 + B.
\end{equation*}
Letting $N = c^1 I - \hat{w} + A I$, we note that $N$ is symmetric 
and Toeplitz as well.

Following the work of \cite{FERREIRA19941227}, we now derive bounds on the eigenvalues of $N$. Define the symmetric, Toeplitz matrix $S$ as the matrix whose first row is given by $\begin{bmatrix} 0 & -w_{M-2} &  -w_{M-3} &\dots &-w_{2} & -w_{1} \end{bmatrix}$ or, more explicitly,
\begin{equation*}
S = \begin{bmatrix}
0 & -w_{M-2} & \dots & -w_{1} \\
-w_{M-2} & 0 & \dots & -w_{2}\\
\vdots & \vdots & \ddots\ & \vdots \\
-w_{1} & -w_{2} & \dots & 0
\end{bmatrix}.
\end{equation*}
Now define the block matrix $C$ as
\begin{equation*}
C = \begin{bmatrix}
N & S \\
S & N 
\end{bmatrix},
\end{equation*}
and note that not only is $C$ a symmetric, Toepltiz matrix but it is also circulant. Since it's circulant, the eigenvalues are given explicitly by 
\begin{align*}
\lambda_j &= \sum\limits_{k = 1}^{2M-2} C_{1k}z_j^{k-1} \\
&=(c^1 + A) -w_1 z_j -w_2 z_j^2- \dots -w_{M-2} z_j^{M-2} \\
& \quad -w_{M-2} z_j^{M} -w_{M-3} z_j^{M+1} - \dots - w_{1} z_j^{2M-3},
\end{align*}
where $z_j = \text{exp}(i\frac{2\pi j}{2M-2})$ and $0 \leq j \leq 2M-3$. 

Let $\mu_1$ and $\mu_{M-1}$ denote the smallest and largest eigenvalues 
of N, respectively. Then using the main result of \cite{FERREIRA19941227} we can 
bound the eigenvalues of $N$ by the eigenvalues of $C$. Specifically, we have
\begin{align*}
(\min\limits_{j \; \text{even}} \lambda_j) + (\min\limits_{j \; \text{odd}} \lambda_j) &\leq 2\mu_1, \\
(\max\limits_{j \; \text{even}} \lambda_j) + (\max\limits_{j \; \text{odd}} \lambda_j) &\geq 2\mu_{M-1}.
\end{align*}
We then see that if we can derive a lower bound for all the $\lambda_j$ then this will also be a lower bound for $\mu_1$. Noting that $z_j^{2M-2} = 1$, we can write
\begin{align*}
\lambda_j &= (c^1 + A) -w_1 z_j -w_2 z_j^2- \dots -w_{M-2} z_j^{M-2} \\[.1cm]
& \quad +z_j^{2M-2} (-w_{M-2} z_j^{-M+2} -w_{M-3} z_j^{-M+3} - \dots - w_{1} z_j^{-1} ) \\[.1cm]
&= c^1+ A - \sum\limits_{k = 2-M}^{M-2} w_k z_j^{k} \\[.1cm]
&= \sum\limits_{k =-M}^{M} w_k + A - \sum\limits_{k = 2-M}^{M-2} w_k z_j^{k}.
\end{align*} 
If we now use the fact that $w_j = \nu(hj) h + O(h^2)$ and $h= \frac{2L}{M}$ then we can rewrite the above as
\begin{align*}
\lambda_j &=  \int^{2L}_{-2L} \nu(x) \; dx + A - \int^{2L}_{-2L} \nu(x) e^{i\frac{j \pi}{2L}x} \; dx + O(h) \\[.2cm]
&=  1 - \int^{2L}_{-2L} \nu(x) e^{i\frac{j \pi}{2L}x} \; dx + O(h) \\[.2cm]
&= 1-  \int^{2L}_{-2L} \nu(x) \cos({\frac{j \pi}{2L}x}) \; dx + O(h) \\[.2cm]
& \to   1-  \int^{2L}_{-2L} \nu(x) \cos({\frac{j \pi}{2L}x}) \; dx  \quad \text{as}\quad  h \to 0. 
\end{align*}

Define $\Lambda_j=1-  \int^{2L}_{-2L} \nu(x) \cos({\frac{j \pi}{2L}x}) \; dx$ and note that since the $\lambda_j$ get arbitrarily close to the $\Lambda_j$, it's sufficient to show that $\Lambda_j$ is bounded away from zero for all $j$. In the special case that $j=0$, we see that $\Lambda_0 = A > 0$. 

For $j \geq 1$, denote the symbol of $\mathcal{L}$ as $\hat{\mathcal{L}}$ so that we can write 
\begin{align*}
\Lambda_j  &= \hat{\mathcal{L}}(\frac{j \pi}{2L}) + \int_{|x| \geq 2L} \nu(x) \cos({\frac{j \pi}{2L}x}) \; dx \\[.2cm]
&\geq  \hat{\mathcal{L}}(\frac{j \pi}{2L}) - \int_{|x| \geq 2L} \nu(x) \; dx \\[.2cm]
&\geq  \hat{\mathcal{L}}(\frac{j \pi}{2L}) - 2 e^{-\eta L} \\[.2cm]
&=  M(\frac{j \pi}{2L})\frac{(\frac{j \pi}{2L})^2}{1+(\frac{j \pi}{2L})^2} - 2 e^{-\eta L} \\[.2cm]
& \geq c_M \frac{(j \pi)^2}{(2L)^2+(j \pi)^2} - 2 e^{-\eta L} \\[.2cm]
&> 0, 
\end{align*}
where in the second line we've used that the tails are positive, in the third line we've used that $\nu(x)$ is exponentially localized for large values of $L$, in the fourth line we use the lemma \ref{l:decompositionconv},
in the fifth line that $M$ is bounded below by the positive constant $c_M$, and in the sixth line that it's always possible to choose $L$ large enough such that $c_M \frac{(j \pi)^2}{(2L)^2+(j \pi)^2} - 2 e^{-\eta L}$ is positive for all $j \geq 1$. To finish, note that $c_M \frac{(j \pi)^2}{(2L)^2+(j \pi)^2} - 2 e^{-\eta L}$ is an increasing function of $j$ so that $\Lambda_j$ is bounded below by $ c_M \frac{(\pi)^2}{(2L)^2+(\pi)^2} - 2 e^{-\eta L}$ for all $j \geq 1$. Defining $\Lambda_{\min} = \min \{A, \; c_M \frac{(\pi)^2}{(2L)^2+(\pi)^2} - 2 e^{-\eta L}\}$, this immediately gives that $\mu_1 \geq  \Lambda_{\min} > 0$ and implies the scheme is stable in the (grid) 2-norm.

\subsubsection{Consistency}
To get a precise bound on the local truncation error, note that since $\nu \in L^1(\R)$, we can apply Holder's Inequality to the integrals in Eqs.~(\ref{e:def_fi}). Doing so yields
\begin{equation*}
|f_2(h)| \leq h^4 \| \nu \|_{L^1(\R)}  \quad , \quad |f_3(h)| \leq h^4 \| \nu \|_{L^1(\R)} \quad, \quad  |f_4(h)| \leq h^2 \| \nu \|_{L^1(\R)},
\end{equation*}
which shows that the local truncation error is at least $O(h^2)$.

\subsubsection{Convergence}
Changing notation slightly, let $U(x)$ be the solution of the problem (DP), given by Eq.~(\ref{e:DP}), $\hat{u}$ be the solution of the corresponding discrete scheme above, and define $E^h_i = U(x_i) -\hat{u}_i$ for all $-\frac{M}{2}+1 \leq i \leq \frac{M}{2}-1$. Denoting the local truncation error by LTE, we have
\begin{align*}
NE^h &= NU - N\hat{u} \\
&= NU + c^2 + c^3 + B - N\hat{u} - c^2 - c^3 - B \\
&= \hat{f} + LTE - \hat{f}\\
&= LTE.
\end{align*}
Inverting the matrix $N$ and applying the properties of grid norms gives
\begin{align*}
\| E^h \|_2 &\leq  \| N^{-1} \|_2 \| LTE \|_2 \\
&\leq  \frac{1}{\Lambda_{\min}} \| LTE \|_2 \\
 &\leq \frac{\sqrt{L}}{\Lambda_{\min}} \| LTE \|_\infty
\end{align*}
Since $LTE = O(h^2)$ and $\Lambda_{\min}$ doesn't depend on $h$, we can take the limit as $h \to 0$ on both sides to conclude that $E^h \to 0$, so that the scheme converges.

\subsection{Whole Real Line Problem}
Consider again the scheme for the whole real line
\begin{equation*}
\sum\limits^M_{j=-M}\bigg([u_i-u_{i - j}]w_j\bigg) + Au_i - u_{-\frac{M}{2}} \; B^1_i - u_{\frac{M}{2}} \; B^2_i = f(x_i),
\end{equation*}
for $-\frac{M}{2}\leq i \leq \frac{M}{2}$. 
As the resulting matrix equation isn't symmetric, we will proceed in a different way than the previous subsection to show that the scheme converges. To do this, in addition to Eqs. (\ref{e:ass1}) and (\ref{e:ass2}), we will assume in this section that $\nu(y)$ is a nonnegative $L^1(\R)$ function such that for all sufficiently large values of $L$ we have that the tails are strictly positive:
\begin{equation}\label{e:ass4}
\int_{|y| > 2L} \nu(y) \; dy > 0.
\end{equation}
This contrasts with the Dirichlet problem in that we do not allow the kernel to take possibly  negative values.

Because of Corollary \ref{c:invertible_con}, we will assume that $f$ has been chosen such that the solution $u(x)$ of the (UP) problem satisfies  $|u(x)| \leq \frac{C}{|x|^q}$ for all sufficiently large $x$ and some constants $C,q > 0$. We will then define the decay function as $g(x):= \frac{1}{|x|^q}$.
\subsubsection{Stability}
As before, we first focus on the matrix form of the discrete convolution term. 
We have
\begin{align*}
\sum\limits^M_{j=-M}\bigg([u_i-u_{i-j}]w_j\bigg) &= u_i\sum\limits^M_{j=-M}w_j -\sum\limits^M_{j=-M} u_{i-j}w_j \\[.2cm]
&= u_i\sum\limits^M_{j=-M}w_j -\sum\limits_{|i-j| \leq \frac{M}{2}} u_{i-j}w_j -\sum\limits_{|i-j| > \frac{M}{2}} u_{i-j}w_j \\[.2cm]
&= u_i\sum\limits^M_{j=-M}w_j 
-\sum\limits_{j= i- \frac{M}{2}}^{i+ \frac{M}{2}} u_{i-j}w_j 
-\sum\limits^{i- \frac{M}{2}-1}_{j=-M} u_{i-j}w_j 
-\sum\limits_{j=i+ \frac{M}{2}+1}^{M} u_{i-j}w_j \\[.2cm]
&= u_i\sum\limits^M_{j=-M}w_j 
-\sum\limits_{j= i- \frac{M}{2}}^{i+ \frac{M}{2}} u_{i-j}w_j 
-u_{-\frac{M}{2}}|L|^q\sum\limits^{i- \frac{M}{2}-1}_{j=-M} g_{i-j}w_j 
\\
&\quad -u_{\frac{M}{2}}|L|^q\sum\limits_{j=i+ \frac{M}{2}+1}^{M} g_{i-j}w_j .
\end{align*}
Renaming the above quantities as follows
\begin{equation*}
c^1=\sum\limits^M_{j=-M}w_j  \quad, \quad c^2_i=|L|^q\sum\limits^{i- \frac{M}{2}-1}_{j=-M} g_{i-j}w_j \quad , \quad c^3_i=|L|^q\sum\limits_{j=i+ \frac{M}{2}+1}^{M} g_{i-j}w_j,
\end{equation*} 
we get
\begin{align*}
\sum\limits^M_{j=-M}\bigg([u_i-u_{i-j}]w_j\bigg) &= u_i c^1 - u_{-\frac{M}{2}}c^2_i - u_{\frac{M}{2}}c^3_i
-\sum\limits_{j= i- \frac{M}{2}}^{i+ \frac{M}{2}} u_{i-j}w_j \\[.2cm]
&= u_i c^1-u_{-\frac{M}{2}}c^2_i - u_{\frac{M}{2}}c^3_i - \sum\limits_{k=-\frac{M}{2}}^{\frac{M}{2}} u_{k}w_{i-k}.
\end{align*}
The last equality is obtained by the change of variables $k=i-j$. 
Letting $\hat{u}$ denote the vector
\begin{equation*}
\begin{bmatrix}
u_{-\frac{M}{2}} \\
\vdots \\
u_{\frac{M}{2}}
\end{bmatrix},
\end{equation*}
we can rewrite the above equation in matrix form as
\begin{align*}
&c^1 \hat{u} - u_{-\frac{M}{2}}c^2  - u_{\frac{M}{2}}c^3 - \hat{w}\hat{u} \\[.2cm]
&= \big(c^1 I- [c^2 \; 0 \; \dots \;  0 \; c^3] - \hat{w} \big)\hat{u},
\end{align*}
where $\hat{w}$ is the matrix given by
\begin{equation*}
\begin{bmatrix}
w_0 & w_{-1} & \dots & w_{-M} \\
w_1 & w_0 & \dots & w_{-M}\\
\vdots & \vdots & \ddots\ & \vdots \\
w_{M} & w_{M-1} & \dots & w_0
\end{bmatrix},
\end{equation*}
$c^2, c^3$ are just vectorized versions of $c^2_i, c^3_i$, and 
$[c^2 \; 0 \; \dots \;  0 \; c^3]$ 
has enough zero vectors to make the multiplication well-defined.  
We can then write the numerical scheme for the whole real line problem as
\begin{equation*}
 \big(c^1 I + AI- [(c^2 + B^1)  \; \; 0 \; \dots \;  0 \; \; (c^3 + B^2)] - \hat{w} \big)\hat{u} = \hat{f}
\end{equation*}
which we note is neither symmetric nor Toeplitz like the Dirichlet case.

Defining $N := c^1 I + AI- [(c^2 + B^1)  \; \; 0 \; \dots \;  0 \; \; (c^3 + B^2)] - \hat{w}$ and $Z := I - N$, we now want to show that the $L^\infty$ norm of $Z$ is strictly less than one for all values of $h$. We will then able to bound the $L^\infty$ norm of $N^{-1}$ in terms of the norm of $Z$ via the corresponding Neumann series 
\[N^{-1} = \sum_{n=0}^{\infty} Z^n.\]
Further, since 
\begin{align*}
\|Z\|_\infty = \max\limits_{0 \leq k \leq M} \sum_{j=0}^M |Z_{kj}|
\end{align*}
it's enough to show that the $L^1$ norm of each row is strictly less than one.

For simplicity, we first derive three inequalities which will be needed. To begin, note that
\begin{align*}
A + c^1 &= \int_{|y| \geq L_W} \nu(y) \; dy + \sum_{j=-M}^M w_j \\
&= \int_{|y| \geq L_W} \nu(y) \; dy + \sum_{j=-M}^M \int_{h \leq |y| \leq L_W} T(y -x_j) \nu(y) \; dy + 2f_1(h)\\
&= \int_{|y| \geq L_W} \nu(y) \; dy +  \int_{h \leq |y| \leq L_W} \bigg(\sum_{j=-M}^M T(y -x_j)\bigg) \nu(y) \; dy +2f_1(h) \\
&= \int_{|y| \geq L_W} \nu(y) \; dy +  \int_{h \leq |y| \leq L_W} \nu(y) \; dy + 2f_1(h) \\
&\leq 1,
\end{align*}
where the last line is given by using that $f_1(h) \leq \int_0^h \nu(y) \; dy$. Also, since $|L|^q g(y) \leq 1$ for all $|y| \geq L$ it follows that
\begin{align*}
c^2_i \leq \sum^{i- \frac{M}{2}-1}_{j=-M} w_j \quad , \quad c^3_i \leq \sum\limits_{j=i+ \frac{M}{2}+1}^{M} w_j
\end{align*}
and, by using that the $w_j$ are even, that 
\[\sum_{j = -k}^{M-k} w_j + c^2_{-\frac{M}{2}+k} + c^3_{-\frac{M}{2}+k} \leq c^1.\]
Finally, define $P_k(L) := \int_{|y| \geq L_W} ( 1- |L|^q g(x_{-\frac{M}{2}+k}-y))\nu(y) \; dy$ and note that
\begin{align*}
P_k(L) &= \int_{|y| \geq L_W} ( 1- |L|^q g(x_{-\frac{M}{2}+k}-y))\nu(y) \; dy \\[.2cm]
&= \int_{|y| \geq L_W} ( 1-\frac{ |L|^q }{|x_{-\frac{M}{2}+k}-y|^q}) \nu(y) \; dy \\[.2cm]
&\geq \int_{|y| \geq 2L_W} ( 1-\frac{ |L|^q }{|x_{-\frac{M}{2}+k}-y|^q}) \nu(y) \; dy \\[.2cm]
&\geq \int_{|y| \geq 2L_W} ( 1-\frac{ |L|^q }{|3L|^q}) \nu(y) \; dy \\[.2cm]
&= (1 - \frac{1}{3^q}) \int_{|y| \geq 2L_W} \nu(y) \; dy.
\end{align*}
Defining $\Lambda_{\min}(L) := (1 - \frac{1}{3^q}) \int_{|y| \geq 2L_W} \nu(y) \; dy$, we note that $\Lambda_{\min}$ is strictly positive and independent of $h$. For reference, we list the three inequalities here as
\begin{gather}
c^1 + A \leq 1 \label{in: 1} \\[.1cm]
\sum_{j = -k}^{M-k} w_j + c^2_{-\frac{M}{2}+k} + c^3_{-\frac{M}{2}+k} \leq c^1 \label{in: 2}\\[.1cm]
P_k  \geq \Lambda_{\min} \label{in: 3}.
\end{gather}

Returning to the the $Z$ matrix, we have for the first row that
\begin{align*}
0 \leq \sum_{j=0}^M |Z_{0j}| & = | 1- c^1 -A + c^2_{-\frac{M}{2}} + B^1_{-\frac{M}{2}} | + \sum_{j = 1}^{M-1} |w_j| + |  c^3_{-\frac{M}{2}} + B^2_{-\frac{M}{2}} + w_{M}| \\
& \leq | 1- c^1 -A| + |c^2_{-\frac{M}{2}} + B^1_{-\frac{M}{2}} | + \sum_{j = 1}^{M-1} |w_j| + |  c^3_{-\frac{M}{2}} + B^2_{-\frac{M}{2}} + w_{M}| \\
& =  1- c^1 -A + c^2_{-\frac{M}{2}} + B^1_{-\frac{M}{2}}  + \sum_{j = 0}^{M} w_j + c^3_{-\frac{M}{2}} + B^2_{-\frac{M}{2}} \\
&\leq 1 - A + B^1_{-\frac{M}{2}} + B^2_{-\frac{M}{2}} \\
& = 1 - \int_{|y| \geq L_W} ( 1- |L|^q g(x_{-\frac{M}{2}}-y))\nu(y) \; dy \\
&= 1- P_0 \\
&\leq 1- \Lambda_{\min},
\end{align*}
where in the second line we used triangle inequality, in third we used inequality~(\ref{in: 1}), in the fourth inequality~(\ref{in: 2}), and in the last line inequality~(\ref{in: 3}). Similarly, for the last row of the $Z$ matrix we have
\begin{align*}
0 \leq \sum_{j=0}^M |Z_{Mj}| & = |c^2_{\frac{M}{2}} + B^1_{\frac{M}{2}} + w_M| + \sum_{j = 1}^{M-1} |w_j| + | 1- c^1 -A +  c^3_{\frac{M}{2}} + B^2_{\frac{M}{2}}| \\
& \leq c^2_{\frac{M}{2}} + B^1_{\frac{M}{2}}  + \sum_{j = 0}^{M} w_j +  1- c^1 -A +  c^3_{\frac{M}{2}} + B^2_{\frac{M}{2}} \\
&\leq 1 - A + B^1_{\frac{M}{2}} + B^2_{\frac{M}{2}} \\
& = 1 - \int_{|y| \geq L_W} ( 1- |L|^q g(x_{\frac{M}{2}}-y))\nu(y) \; dy \\
&= 1- P_M \\
&\leq 1- \Lambda_{\min},
\end{align*}
where the inequalities (\ref{in: 1}), (\ref{in: 2}), and (\ref{in: 3}) were used in the same way as before. By applying the same argument, we have for any row in between the first and last
\begin{align*}
0 &\leq \sum_{j=0}^M |Z_{kj}| \\
& =  |c^2_{-\frac{M}{2}+k} + B^1_{-\frac{M}{2}+k} + w_{-k}| + \sum_{j = -k}^{M-k-1} |w_j| + | 1- c^1 -A| +  |c^3_{-\frac{M}{2}+k} + B^2_{-\frac{M}{2}+k}+w_{M-k}| \\
& \leq  c^2_{-\frac{M}{2}+k} + B^1_{-\frac{M}{2}+k} + \sum_{j = -k}^{M-k} w_j + 1- c^1 -A +  c^3_{-\frac{M}{2}+k} + B^2_{-\frac{M}{2}+k} \\
&\leq 1 - A + B^1_{-\frac{M}{2}+k} + B^2_{-\frac{M}{2}+k} \\
&= 1- P_k \\
&\leq 1- \Lambda_{\min}.
\end{align*}

Since the same bound applies to each of the sums, we must have that
\[ \|Z\|_\infty \leq 1 - \Lambda_{\min} < 1. \]
We then have stability of the $N$ matrix since
\begin{align*}
\|N^{-1}\|_\infty &= \| (I - Z)^{-1}\|_\infty \\
&= \| \sum_{n=0}^\infty Z^n \|_\infty \\
&\leq \sum_{n=0}^\infty \| Z \|_\infty ^n \\
&\leq \sum_{n=0}^\infty (1- \Lambda_{\min}) ^n \\
&= \frac{1}{\Lambda_{\min}}.
\end{align*}

\subsubsection{Consistency}

The above argument establishes the numerical method is stable. We now need to bound the local truncation error in order to get consistency. To this end, we'd like to show that 
\begin{align}\label{e:UPCons}
&\int_{\R} (u(x_i) - u(x_i-y)) \nu(y) \; dy \nonumber \\ 
=&\sum\limits^M_{j=-M}\bigg([u_i-\tilde{u}_{i - j}]w_j\bigg) + Au_i - u_{-\frac{M}{2}} \; B^1_i - u_{\frac{M}{2}} \; B^2_i  \\
& + O(h^2,|L|^{-q}),\nonumber
\end{align}
where we've placed a tilde on the second term in the sum to remind the reader that if $|i-j| > \frac{M}{2}$ then $\tilde{u}_{i-j} = u(\pm L)|L|^q g_{i-j}$. 

We'll do this in four steps. First, decompose the integral on the LHS into four pieces, corresponding to each of the first four terms on the RHS respectively:
\begin{align*}
\int_{-2L}^{2L} (u(x_i) - u(x_i-y)) \nu(y) \; dy  + u(x_i) \int_{|y| \geq 2L} \nu(y) \; dy \\[.2cm]
- \int_{-\infty}^{-2L}u(x_i-y) \nu(y) \; dy - \int_{2L}^{\infty}u(x_i-y) \nu(y) \; dy.
\end{align*}
Second, note that if $x \in [-L,L]$ and $y \in (-2L,2L)^c$ then
\begin{align*}
|u(x -y) - \frac{u(\pm L)}{g(\pm L)} g(x - y) | & \leq  |u(x - y)| + |u(\pm L)\frac{g(x - y)}{g(\pm L)}| \\[.1cm]
& \leq  \frac{C}{|x - y|^q} + \frac{C}{|L|^q} \frac{|L|^q}{|x - y|^q}\\[.2cm]
& \leq  \frac{2C}{|x - y|^q}\\[.2cm]
& \leq  \frac{2C}{|L|^q}.
\end{align*}
Next, by using the previous inequality we get
\begin{align*}
& | \int_{y \geq 2L} u(x_i-y) \nu(y) \; dy  - u_{\frac{M}{2}} \; B^2_i| \\[.2cm]
= \;& | \int_{y \geq 2L} u(x_i-y) \nu(y) \; dy  - \int_{|y| \geq 2L} \frac{u(L)}{g(L)} g(x_i - y) \nu(y) \; dy | \\[.2cm]
\leq \; &   \frac{2C}{|L|^q} \int_{y \geq 2L} \nu(y) \; dy,
\end{align*}
and something similar for the corresponding pair. If we now note that $u(x_i) \int_{|y| \geq 2L} \nu(y) \; dy = A u_i$, then we have the preliminary bound
\begin{align*}
 &u(x_i) \int_{|y| \geq 2L} \nu(y) \; dy - \int_{-\infty}^{-2L}u(x_i-y) \nu(y) \; dy - \int_{2L}^{\infty}u(x_i-y) \nu(y) \; dy\\[.2cm]
 = \; & Au_i - u_{-\frac{M}{2}} \; B^1_i - u_{\frac{M}{2}} \; B^2_i + O(|L|^{-q}).
\end{align*}

Finally, consider the last remaining integral $\int_{-2L}^{2L} (u(x_i) - u(x_i-y)) \nu(y) \; dy$. We have
 \begin{align*}
& |\int_{-2L}^{2L} (u(x_i) - u(x_i-y)) \nu(y) \; dy  - \sum\limits^M_{j=-M}\bigg([u_i-\tilde{u}_{i - j}]w_j\bigg) | \\[.2cm]
\leq \; &  |\int_{-2L}^{2L} (u(x_i) - u(x_i-y)) \nu(y) \; dy  - \sum\limits^M_{j=-M}\bigg([u_i-u_{i - j}]w_j\bigg)|  \\[.2cm]
&+  | \sum\limits^M_{j=-M}\bigg([u_i-u_{i - j}]w_j\bigg) -  \sum\limits^M_{j=-M}\bigg([u_i-\tilde{u}_{i - j}]w_j\bigg) | \\[.2cm]
= \; &  |\int_{-2L}^{2L} (u(x_i) - u(x_i-y)) \nu(y) \; dy  - \sum\limits^M_{j=-M}\bigg([u_i-u_{i - j}]w_j\bigg)|  \\[.2cm]
&+  | \sum\limits^M_{j=-M}\bigg([u_{i - j} - \tilde{u}_{i - j}]w_j\bigg)  |.
\end{align*}
With regards to the last line, note that for fixed $L$ we've already shown in the Dirichlet problem section that the first term is $O(h^2)$; as long as we take $h \to 0$ before changing $L$, this first term will be identically zero. By using our inequality above, the second term can be bounded as
\begin{align*}
 | \sum\limits^M_{j=-M}\bigg([u_{i - j} - \tilde{u}_{i - j}]w_j\bigg)  |
\leq  \frac{2C}{|L|^q} \int_{-2L}^{2L} \nu(y) \; dy,
\end{align*}
which we note is independent of $h$. 

Collecting results, this shows Eq.~(\ref{e:UPCons}) holds true and the scheme is consistent for all sufficiently large $L$. Further, for all sufficiently large $L$, as $h \to 0$ we have that the pointwise error term is bounded above by
\begin{align*}
&\frac{2C}{|L|^q} \int_{|y| \geq 2L} \nu(y) \; dy + \frac{2C}{|L|^q} \int_{-2L}^{2L} \nu(y) \; dy\\[.2cm]
= \; & \frac{2C}{|L|^q} .
\end{align*}

\subsubsection{Convergence}
Let $U(x)$ be the solution of the problem (UP), given by Eq.~(\ref{e:UP}), $\hat{u}$ be the solution of the corresponding discrete scheme above, and define $E^h_i = U(x_i) -\hat{u}_i$ for all $-\frac{M}{2} \leq i \leq \frac{M}{2}$. Denoting the local truncation error by LTE, we have
\begin{align*}
NE^h &= NU - N\hat{u} \\
&= \hat{f} + LTE - \hat{f}\\
&= LTE.
\end{align*}
Inverting the matrix $N$ and applying the properties of grid norms gives
\begin{align*}
\| E^h \|_\infty &\leq  \| N^{-1} \|_\infty \| LTE \|_\infty \\
&\leq  \frac{1}{\Lambda_{\min}} \| LTE \|_\infty.
\end{align*}
Since $LTE = O(h^2,|L|^{-q})$ and $\Lambda_{\min}$ doesn't depend on $h$, we can take the limit as $h \to 0$ on both sides to conclude that
\begin{align*}
\lim\limits_{h \to 0} \| E^h \|_\infty &\leq   \frac{1}{\Lambda_{\min}} \frac{2C}{|L|^q}.
\end{align*}
Notice however that because the bound for the smallest eigenvalue of matrix $N$, $\Lambda_{\min}$, decays exponentially with $L$, 
the convergence of our scheme as $L$ goes to infinity is not established. 
Nonetheless, our numerical examples show that for fixed $L$ the algorithm does converge at order $\rmO(h^2)$,
which is to be expected when $h^2> L^{-q}$.
We suspect that because the right hand side, $f$, decays algebraically and satisfies the compatibility conditions
 $\langle f,1 \rangle = \langle f, x \rangle =0$, just like in the analytical setting (see Section \ref{s:diffusive_operators}),
  the solution avoids small wavenumbers,  allowing for the convergence of the scheme for large values of $L$.

\section{Numerical illustrations}
Here we consider a series of examples to illustrate the usefulness of
our numerical schemes.

\subsection{Dirichlet Problem}
\label{s:Dirichlet}
Consider the Dirichlet problem
\begin{equation}
\begin{cases}
\mathcal{L} \ast u(x) = f(x), & x \in (-L, L), \\ 
u(x) = \text{sech}(x), & x \in (-L,L)^c,
\end{cases}
\end{equation}
with
\begin{equation*}
\nu(y) = \frac{1}{2}e^{-|y|}, \quad f(x) = \text{sech}(x)- \frac{1}{2}e^{-x}\log(1+e^{2x}) - \frac{1}{2}e^{x}\log(1+e^{2x}) + xe^x.
\end{equation*}
Then by direct calculation, we have that
\begin{gather*}
f_1(h) = \frac{1}{2}h^{-2} \bigg[ 2 - e^{-h}(h^2 + 2h +2)\bigg] = \frac{h}{6} + O(h^2), \\[.2cm]
f_2(h) = \frac{1}{3} h^5 + O(h^6) \quad , \quad  f_3(h) = \frac{1}{5} h^5 + O(h^6) \quad , \quad f_4(h) = 2 h^2 + O(h^3).
\end{gather*}
Using remark \ref{r:order}, 
we should generically expect the scheme to converge at rate $O(h^2)$.
In this particular case, $\nu(y)$ has the antiderivatives
\begin{equation*}
F(y) = \frac{1}{2}e^{-|y|} \quad , \quad F'(y) = -\frac{1}{2} \text{sign} (y) e^{-|y|} \quad , \quad F''(y) = \frac{1}{2}e^{-|y|}.
\end{equation*}
As we show in Appendix \ref{a:weights_calc}, all of the $w_j$ can now be
calculated as follows
\begin{equation*}
w_j =
\begin{cases}
f_1(h) - F'(x_1) + \frac{1}{h}[F(x_{2}) -F(x_{1})], & |j|= 1 \\[.2cm]
\frac{1}{h}[F(x_{j+1}) - 2F(x_{j})+F(x_{j-1})], & 1< |j| <M \\[.2cm]
F'(x_M) + \frac{1}{h}[F(x_{M-1}) -F(x_{M})], & |j| = M
\end{cases}
\end{equation*}
and $w_0 = 0$. We also have that the integral $A$, defined in \eqref{e:def_A},
can be directly computed and is given by
\[A = e^{-L_W}.\]
The integral $B_i$, defined in \eqref{e:def_Bi}, can also be calculated directly and is given by
\begin{equation*}
B_i = \frac{1}{2}e^{x_i} \log(e^{-2L_W}+e^{2x_i})-e^{x_i}x_i + \frac{1}{2}e^{-x_i} \log(e^{-2L_W}+e^{-2x_i})+e^{-x_i}x_i.
\end{equation*}

We then have all the necessary quantities to implement the numerical
scheme for the Dirichlet problem
introduced in section \ref{ss:scheme_Dirichlet}. 
First note that the true solution of
this problem is in fact $u(x) = \text{sech}(x)$; this can be confirmed
by a straightforward integration. Fig.~\ref{Fig: Fig1}(a) shows a plot
of the solution $u(x)$ and the forcing function $f(x)$. 
Fig.~\ref{Fig: Fig1}(b) shows the $L^\infty$ error between the numerical solution
and the true solution for varying values of $L$ and $h$. We see that
for fixed $L$ the scheme does indeed converge at an $O(h^2)$ rate.

\begin{figure}
\subfloat[]{\includegraphics[width=.5\textwidth]{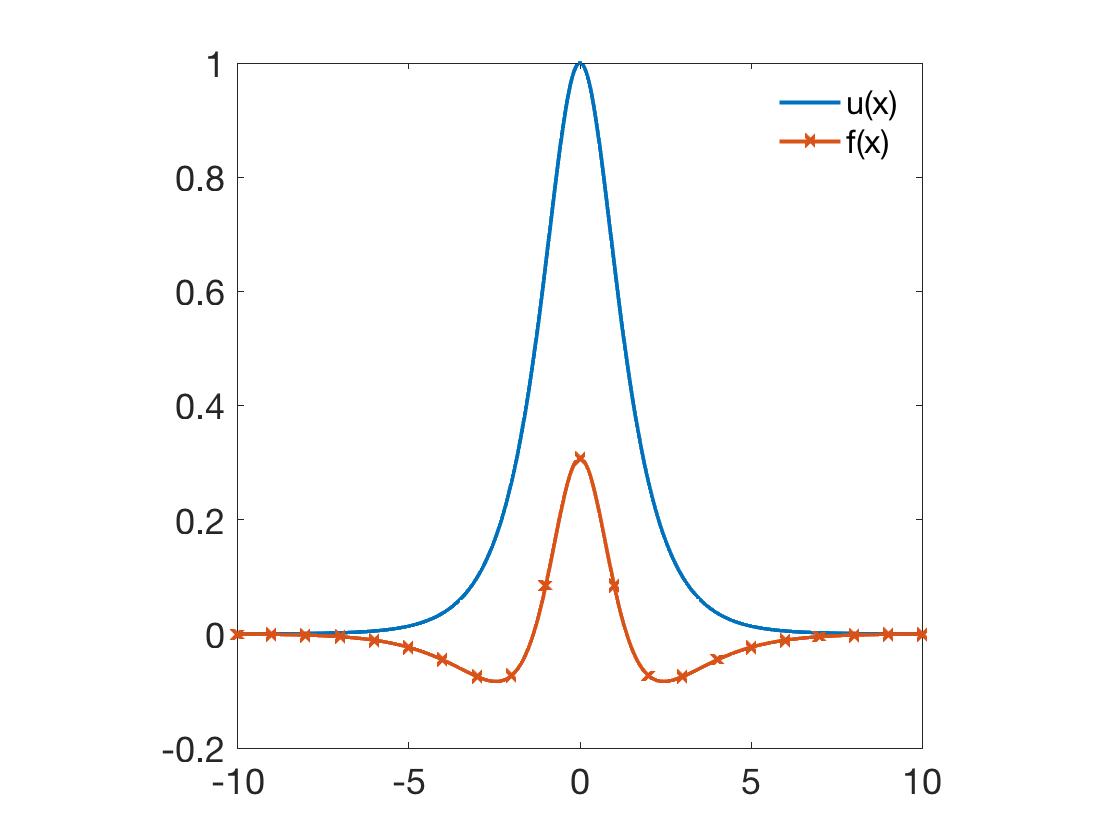}}
\subfloat[]{\includegraphics[width=.5\textwidth]{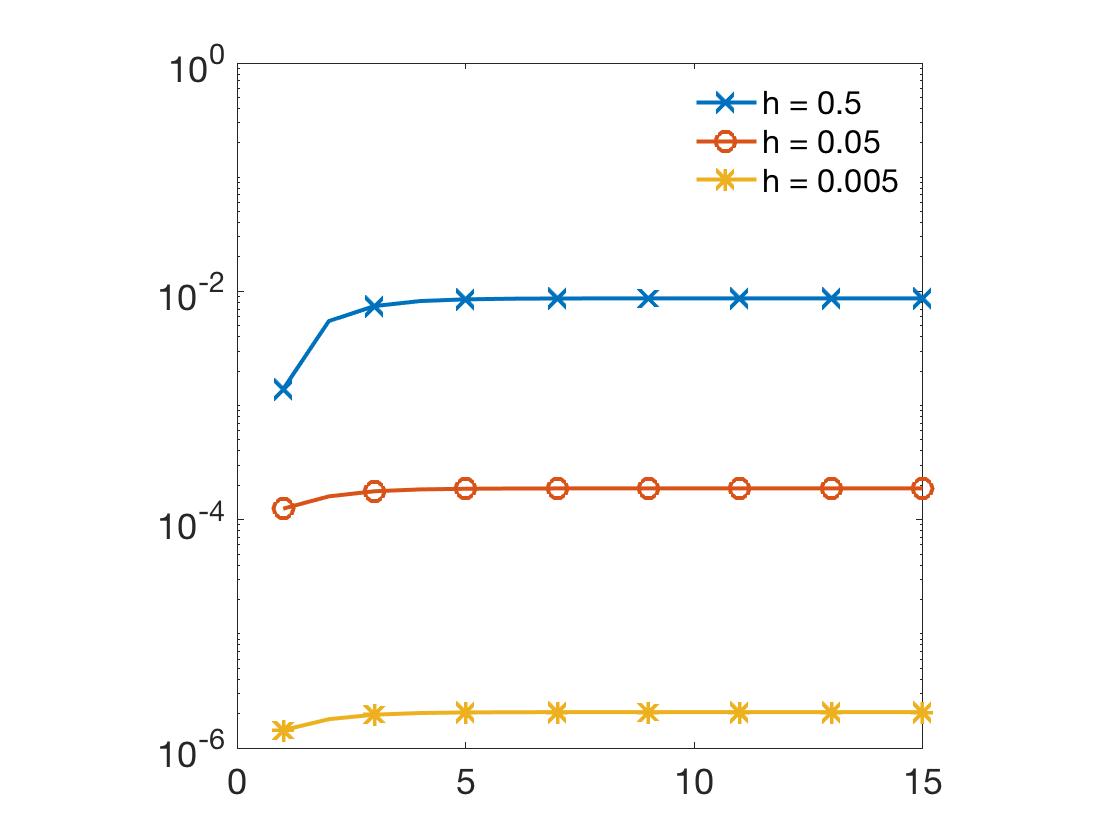}}
\caption{(a) Plot of the solution $u(x) = \text{sech}(x)$ and the
corresponding forcing function $f(x)$. (b) Plot of he $L^\infty$
error between the numerical solution and the true solution for
varying values of $L$; each curve represents a different choice in
the spatial step size $h$.}\label{Fig: Fig1}
\end{figure}

\subsection{Whole Real Line Problem}
\label{s:real_line}
Consider the extended Dirichlet problem
\begin{equation*}
\mathcal{L} \ast u(x) = f(x),\quad x \in \mathbb{R}
\end{equation*}
with kernel
\begin{equation*}
\nu(y) = \frac{1}{2}e^{-|y|}.
\end{equation*}

As the quantities $f_k$, $w_j$ and $A$ have been computed in the 
previous section, to implement the numerical scheme introduced in section 
\ref{ss:scheme_whole_real_line}, we only need to compute the quantities $B^1_i$ 
and $B^2_i$ that are defined in \eqref{e:def_B1i_B2i}.
For this, let us assume for the moment that
\[g(x) = \frac{1}{|x|^p}\]
for some $p > 0$. We then have
\begin{align*}
B_i^1 &= \frac{1}{g(-L)} \int_{-\infty}^{-L_W}g(x_i -y) \nu(y) \; dy \\[.2cm]
&= L^p \int_{-\infty}^{-L_W}\frac{1}{|x_i-y|^p} \; \frac{1}{2}e^{-|y|} \; dy \\[.2cm]
&= L^p \int_{-\infty}^{-L_W}\frac{1}{(x_i-y)^p} \; \frac{1}{2}e^{y} \; dy
\end{align*}
where the last equality is obtained by noting that $x_i > y$ for all
$y \in [-\infty,-L_W]$. Doing two changes of variables and simplifying
yields that the above is equal to
\begin{equation*}
L^p \; \frac{1}{2}e^{x_i} \; \frac{1}{(L_W+x_i)^{p-1}} \int_{1}^{\infty}\frac{1}{z^p} \; e^{-(L_W+x_i)z} \; dz.
\end{equation*}
This last integral is exactly of the form of the generalized
exponential integral function $E_p$. We note that this function can be
computed quickly to a given accuracy and there are many public codes
for doing exactly this. Thus, we have
\begin{equation*}
B_i^1 = L^p \; \frac{1}{2}e^{x_i} \; \frac{1}{(L_W+x_i)^{p-1}} E_p(L_W+x_i).
\end{equation*}
Likewise, it can be shown that
\begin{equation*}
B_i^2 = L^p \; \frac{1}{2}e^{-x_i} \; \frac{1}{(L_W-x_i)^{p-1}} E_p(L_W-x_i).
\end{equation*}

We've then shown that if there exists a solution $u(x)$ that decays
algebraically with order $p_0$ then the above numerical scheme 
should give a good approximation by setting $p=p_0$ in the definition of $g$.
To demonstrate this, we will apply the above scheme to a known 
solution. In particular, let
\begin{equation*}
f(x) = \frac{1}{1+x^2} - \frac{1}{2} \; \frac{1}{1+(x-a)^2} - \frac{1}{2} \; \frac{1}{1+(x+a)^2}
\end{equation*}
for some constant $a>0$ and note that
\begin{equation*}
\int^\infty_{-\infty} f(x) \; dx = 0 \quad , \quad \int^\infty_\infty x f(x) \; dx = 0.
\end{equation*}
Hence we know a corresponding solution $u(x)$ will exist. In fact, for
this particular problem, it can be shown by direct substitution that
\begin{equation*}
u(x) = f(x) - \int_{-\infty} ^x \int_{-\infty}^w f(y) \; dy \; dw
\end{equation*}
is a solution. Integrating directly gives
\begin{align*}
u(x) =& \; \bigg[ \frac{1}{1+x^2} - \frac{1}{2} \; \frac{1}{1+(x-a)^2} -
\frac{1}{2} \; \frac{1}{1+(x+a)^2} \bigg] + \frac{1}{2} \; \bigg[ \log(x^2+1) \\[.2cm]
& - \frac{1}{2} \; \log\big((x^2-a^2)^2+2(x^2 + a^2) + 1\big) \bigg] - \frac{1}{2}x\bigg[2\tan^{-1}(x) \\[.2cm]
& - \tan^{-1}(x+a) + \tan^{-1}(x-a)\bigg] +\bigg[ \frac{a}{4}\pi - \frac{a}{2} \tan^{-1}\big(\frac{1 + x^2 - a^2}{2a}\big) \bigg].
\end{align*}
Letting $a=1$, we have that
\begin{align*}
f(x) &= -\frac{3x^2 - 2}{x^6 + x^4 +4x^2 + 4} \\[.2cm]
& \sim  - \frac{3}{x^4}
\end{align*}
and by Taylor expanding about $\infty$ it can be shown that
\begin{align*}
u(x) & \sim  \frac{1}{2x^2}
\end{align*}
as $|x| \to \infty$. Choosing $p=2$, all quantities in the numerical
scheme have been computed and it can now be
implemented. Fig.~\ref{Fig: Fig2}(a) shows a plot of the true solution
$u(x)$ and the forcing function $f(x)$. Fig.~\ref{Fig: Fig2}(b) shows
the $L^\infty$ error between the numerical solution and the true
solution for varying values of $L$ and $h$. We see that for sufficiently 
large $L$ the scheme seems to converge at an $O(h^2)$ rate as well.

\begin{figure} 
\subfloat[]{\includegraphics[width=.5\textwidth]{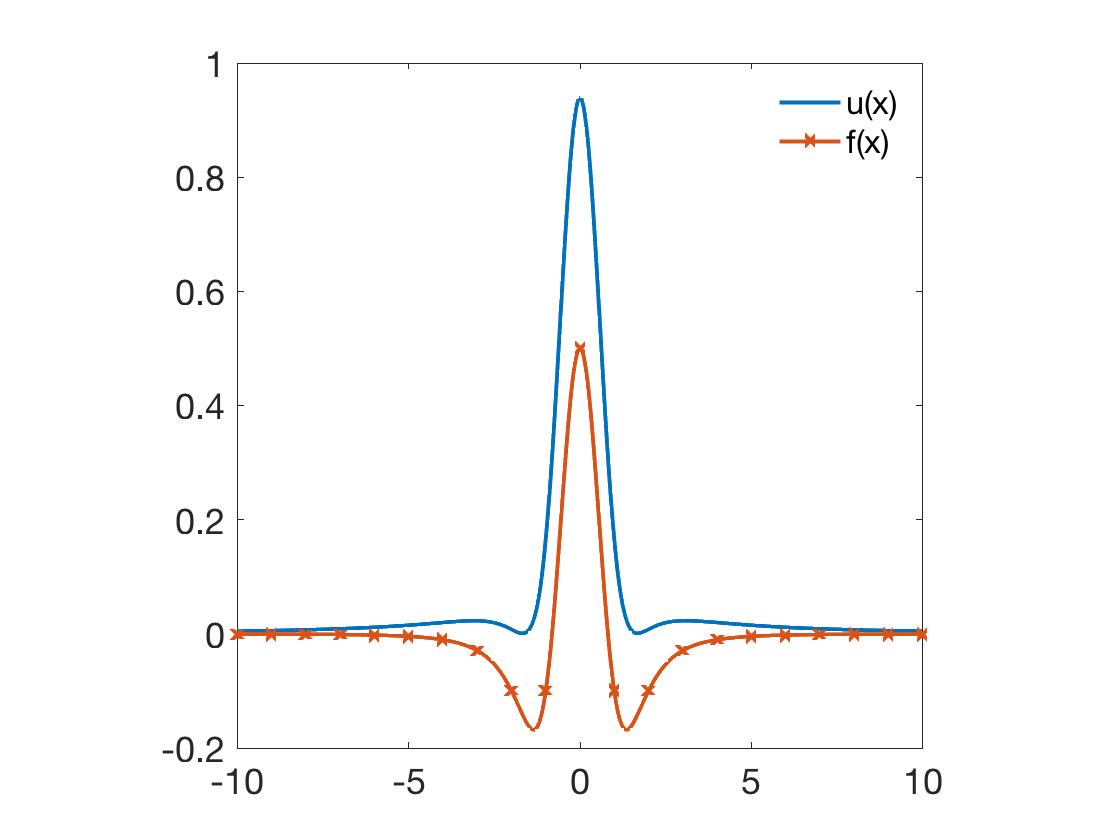}}
\subfloat[]{\includegraphics[width=.5\textwidth]{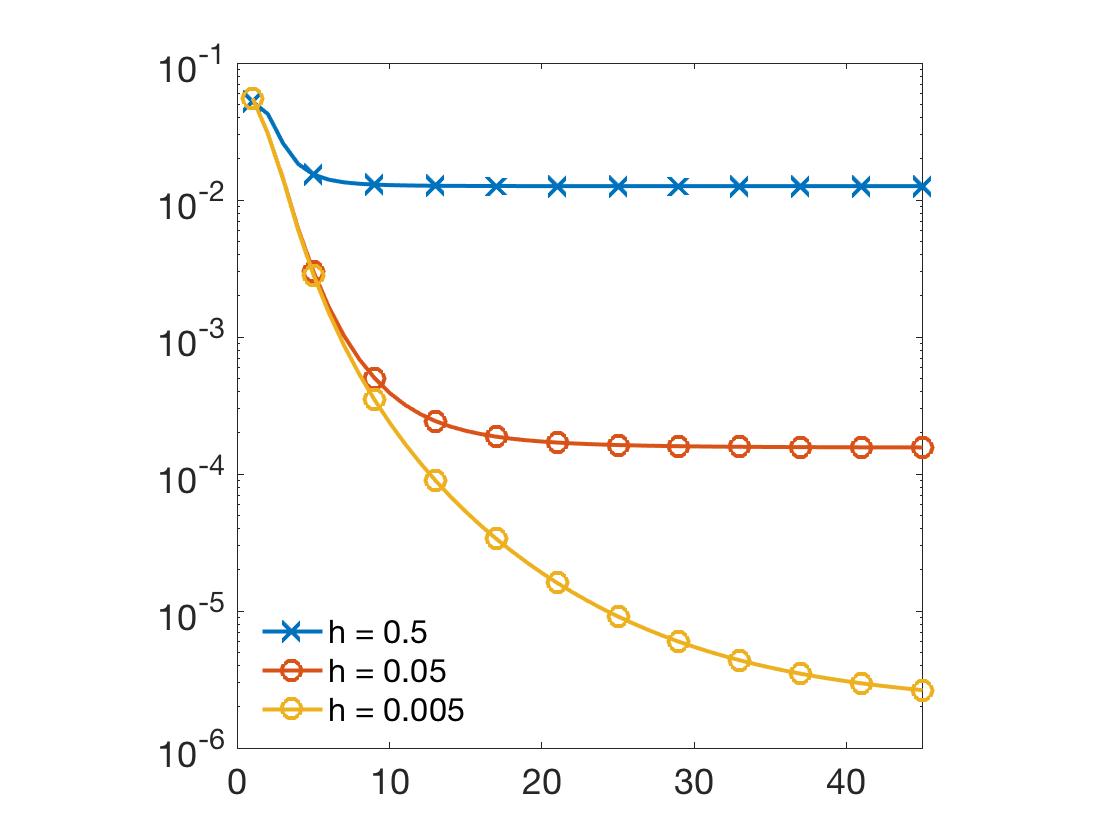}}
\caption{(a) Plot of the algebraically decaying solution $u(x)$ and
 the corresponding forcing function $f(x)$. (b) Plot of he $L^\infty$
 error between the numerical solution and the true solution for
 varying values of $L$; each curve represents a different choice in
 the spatial step size $h$.}\label{Fig: Fig2}
\end{figure}

\subsection{Neumann Problem}
\label{s:Neumann}
Here we consider the Neumann problem
\begin{equation*}
\mathcal{L} \ast u(x) = \bar{f}(x), \quad x \in \mathbb{R}
\end{equation*}
where $\nu(x) = \frac{1}{2} e^{-|x|}$ and
\begin{equation*}
\bar{f}(x) =
\begin{cases}
x^2 -\frac{2}{3}, \quad &x \in (-1,1), \\[.2cm]
\frac{1}{x^4}, \quad &x \in (-1,1)^c.
\end{cases}
\end{equation*}
In this case, it can be shown that the corresponding solution is given
by
\begin{equation*}
u(x) =
\begin{cases}
x^2 -\frac{(x^2-3)(x-1)(x+1)}{12} - \frac{5}{6}, \quad &x \in (-1,1), \\[.2cm]
\frac{1}{x^4} - \frac{1}{6x^2}, \quad &x \in (-1,1)^c.
\end{cases}
\end{equation*}
Note that both $\bar{f}$ and $u(x)$ decay algebraically at the same
rate as before; namely, $\bar{f} \sim \frac{1}{x^4}$ and $u(x) \sim\frac{1}{x^2}$.
Hence, we can apply the numerical method from the
previous section without change. Further note that $\bar{f}$ and
$u(x)$ are not continuous nor differentiable at $x = \pm 1$; see
Fig.~\ref{Fig: Fig3}(a). In deriving the numerical schemes from
previous sections, we implicitly used that the solution $u(x)$ was
many times differentiable. This was done not only to derive formulas
but also to get the $O(h^2)$ truncation error. Since for this
particular example differentiability doesn't hold, we might expect that
the order of convergence of the scheme is no longer $O(h^2)$. Indeed,
this is the case as Fig.~\ref{Fig: Fig3}(b) shows. Instead, it appears
the scheme converges with rate $O(h)$ for the various values of $L$.

\begin{figure}
\subfloat[]{\includegraphics[width=.5\textwidth]{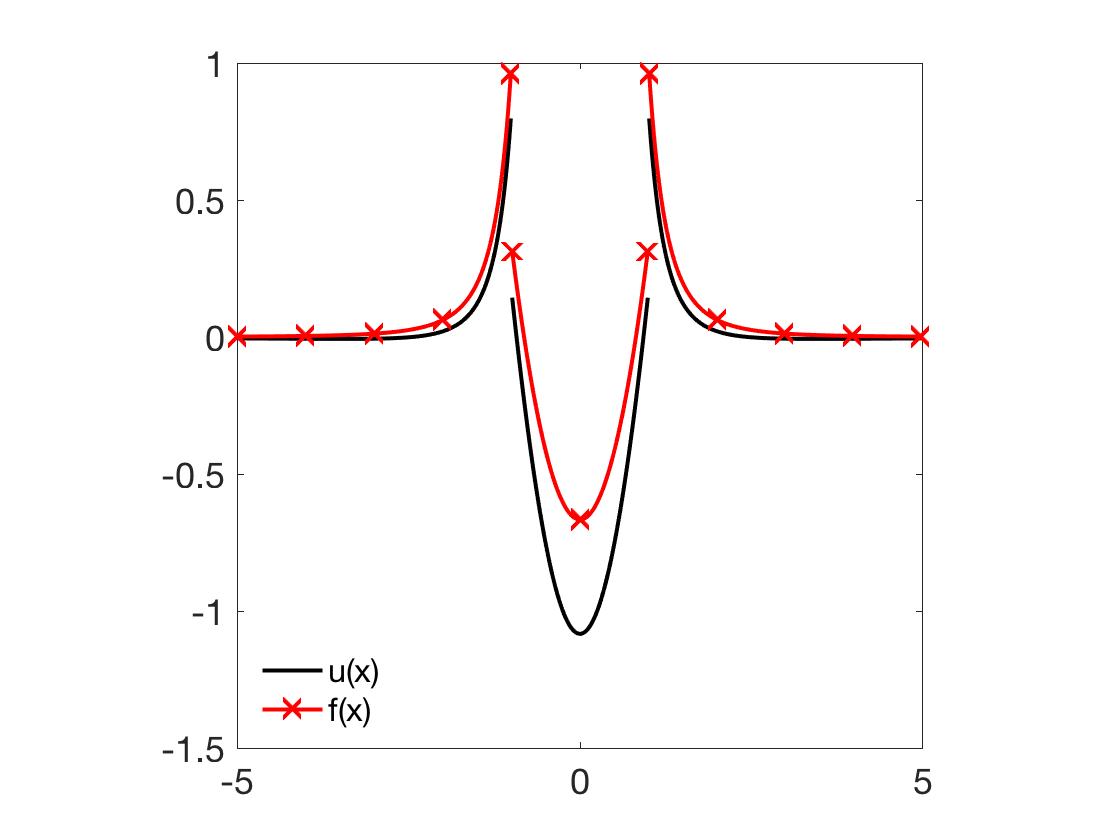}}
\subfloat[]{\includegraphics[width=.5\textwidth]{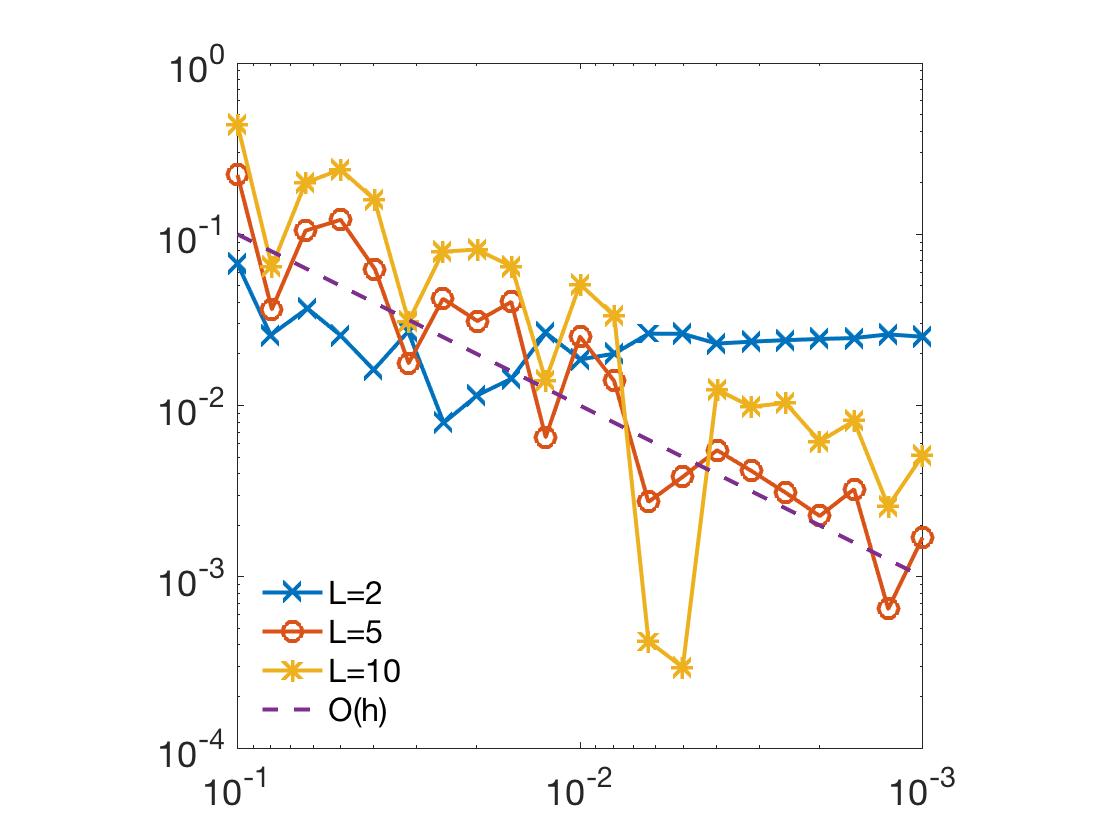}}
\caption{(a) Plot of the solution $u(x)$ and the corresponding forcing
 function $f(x)$. (b) $L^\infty$ error between the numerical solution
 and the true solution for varying values of $h$; unlike before, each
 curve represents a different choice in the computational domain
 $L$.}\label{Fig: Fig3}
\end{figure}

\subsection{Comparison of the Boundary Conditions}
In this section we'd like to test how the different boundary
conditions, and their corresponding numerical schemes, compare in
solving the problem on the whole real line. With this in mind we
consider the problem
\begin{equation*}
\mathcal{L} \ast u(x) = f(x),\quad x \in \mathbb{R},
\end{equation*}
with kernel
\begin{equation*}
\nu(y) = \frac{1}{2}e^{-|y|}
\end{equation*}
and two different forcing functions. In the first case we take
\begin{equation*}
f(x) = \text{sech}(x)- \frac{1}{2}e^{-x}\log(1+e^{2x}) - \frac{1}{2}e^{x}\log(1+e^{2x}) + xe^x
\end{equation*}
for which we know the solution is $u(x) = \sech(x)$.

Fig.~\ref{Fig: Fig4} shows the $L^\infty$ error between the numerical
solution and the true solution for the different boundary
conditions. Fig.~\ref{Fig: Fig4}(a) shows the whole real line
numerical scheme where we've used the asymptotic decay rate of
$\frac{1}{x^2}$ outside of $(-L,L)$. Fig.~\ref{Fig: Fig4}(b)
corresponds to the Dirichlet problem with homogeneous boundary
conditions: $g(x) = 0$ outside $(-L,L)$. Finally,  
Fig.~\ref{Fig: Fig4}(c) corresponds to the Neumann problem with 
homogeneous boundary conditions: $g(x) = 0$ outside $(-L,L)$. 
To be clear, we're setting up the Neumann problem on a uniform 
grid in $(-2L, 2L)$ so that in order to use the previous numerical 
scheme we have to select $L_W  \geq 4L$. Although this requires 
roughly twice the computational cost of the other two methods it 
nevertheless compares the effectiveness of the boundary conditions. 
As the solution $u(x)$ decays exponentially, we note that enforcing
homogeneous boundaries condition is consistent with the original problem 
when $L$ is large enough.
With this in mind, it's clear that for large enough $L$ any of the
three schemes retains the $O(h^2)$ converge rate. 

As shown in Fig.~\ref{Fig: Fig4}(c) , the Neumann formulation requires
larger value of $L$ to obtain a similar convergence behavior than the other 
formulations. It is a consequence of $\bar{f}$ being discontinuous even for large 
value of $L$. However the Neumann formulation is still expected to converge with 
a $O(h^2)$ rate as the function $\bar{f}$ will appear continuous to machine 
precision when $L$ is large enough.
\begin{figure}
\centering
\subfloat[]{\includegraphics[width=.5\textwidth]{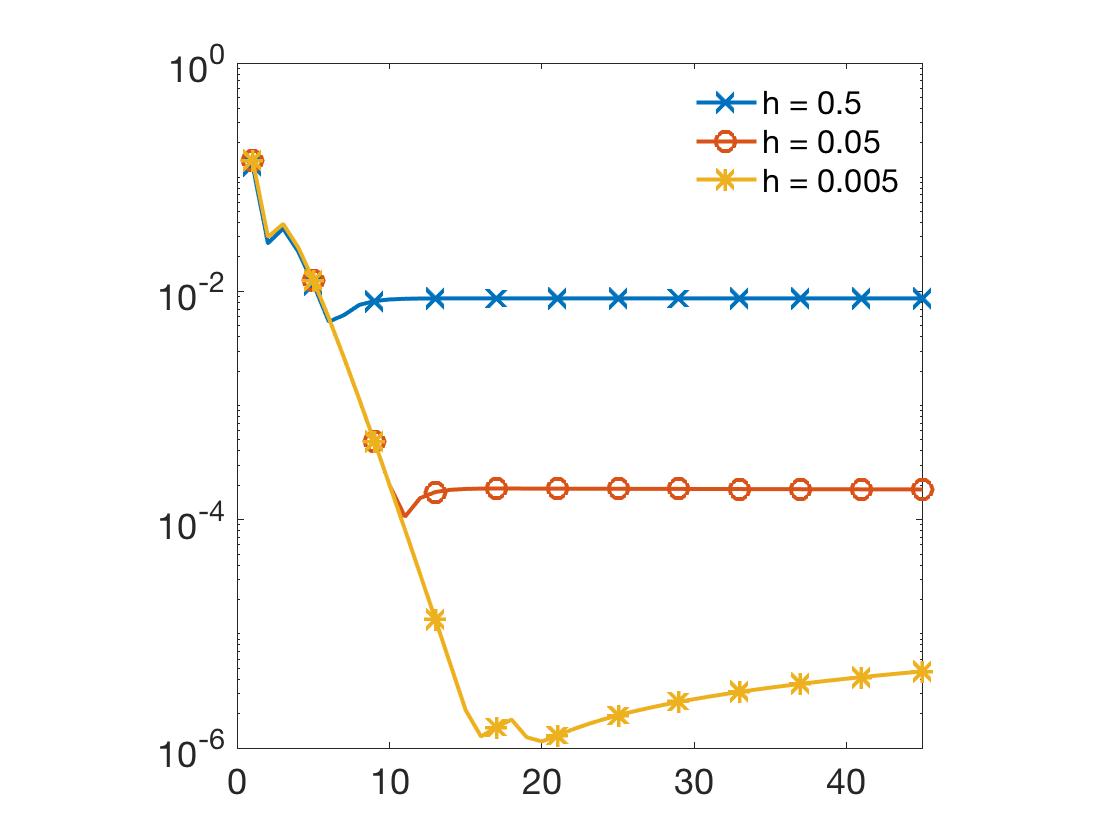}}\\
\subfloat[]{\includegraphics[width=.5\textwidth]{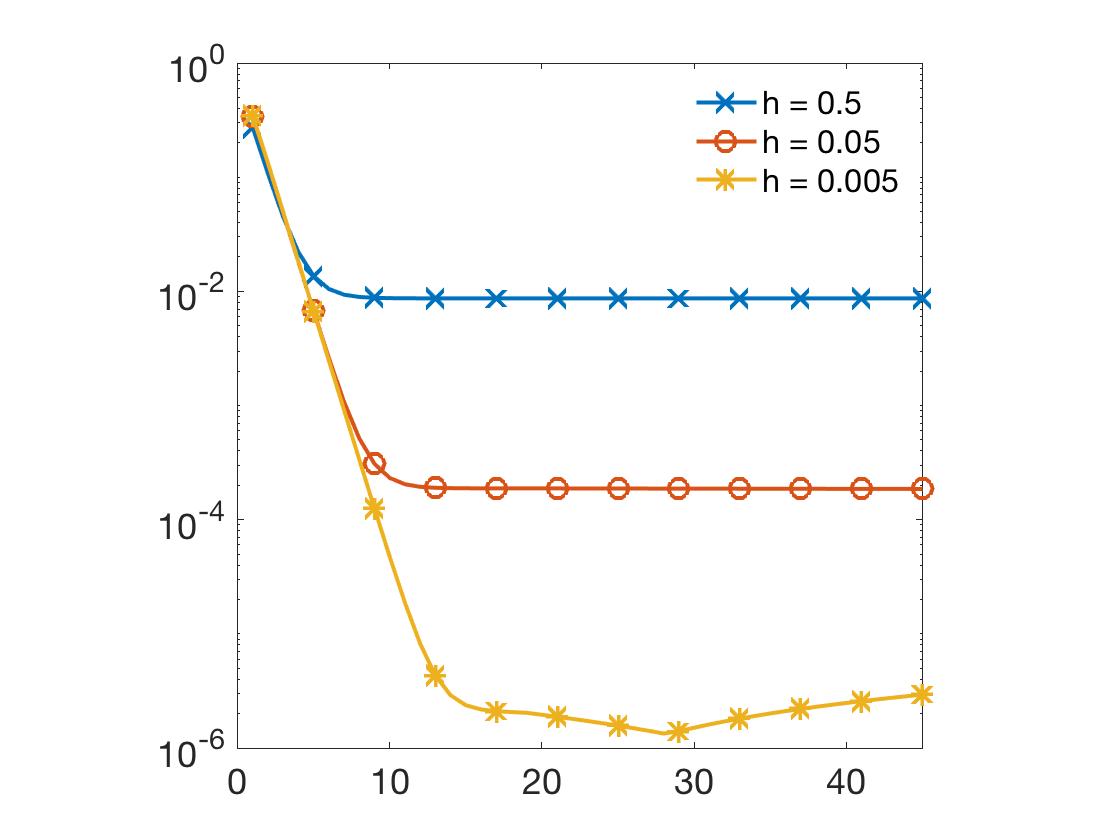}}
\subfloat[]{\includegraphics[width=.5\textwidth]{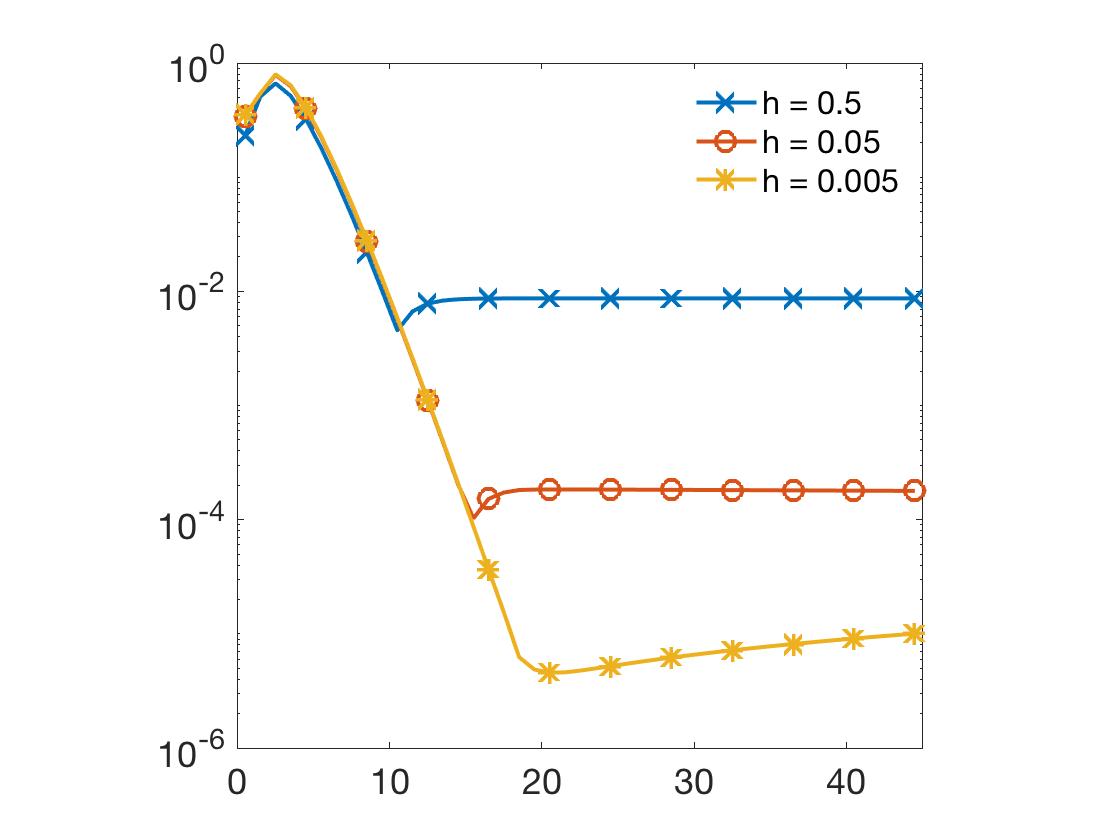}}
\caption{$L^\infty$ error between the numerical solution and the true
 solution for varying values of $L$; each curve represents a
 different choice in the spatial step size $h$. Panel (a) shows the
 results when the algebraic decay is taken into account. Panel (b)
 corresponds to homogeneous Dirichlet boundary conditions. Panel (c)
 corresponds to homogeneous Neumann boundary conditions where we've
 taken $\tilde{L} = \frac{L}{2}$.}\label{Fig: Fig4}
\end{figure}

In the second case we take
\begin{equation*}
f(x) = \frac{1}{1+x^2} - \frac{1}{2} \; \frac{1}{1+(x-a)^2} -
\frac{1}{2} \; \frac{1}{1+(x+a)^2}
\end{equation*}
with $a = 1$ for which we also know the solution. Fig.~\ref{Fig: Fig5}
is the companion figure to Fig.~\ref{Fig: Fig4}.  Fig.~\ref{Fig:
  Fig5}(a) shows the whole real line numerical scheme where we've used
the asymptotic decay rate of $\frac{1}{x^2}$ outside of
$(-L,L)$. Fig.~\ref{Fig: Fig5}(b) corresponds to the Dirichlet problem
with homogeneous boundary conditions: $g(x) = 0$ outside
$(-L,L)$. Finally,  Fig.~\ref{Fig: Fig5}(c) corresponds to the Neumann 
problem with homogeneous boundary conditions, $g(x) = 0$ outside 
$(-L,L)$, and we computed this numerically in the same way as was 
described in the previous paragraph. 
As the use of exact solutions to set boundary conditions is not feasible and 
realistic for physical applications, we use homogeneous conditions for the 
Dirichlet and Neumann problems. It allows us to compare the behavior of the 
three schemes when the solutions is not exponentially decaying and that only its 
order of algebraic decay is known for large $x$.
Unlike the previous case the
situation here is much different. Namely, the whole real line method
is by far more accurate than either of the other two. For the
Dirichlet condition it's because of the slow algebraic decay of the
solution; on the domains considered $u(x)$ does not fall below
$10^{-4}$, making this a lower bound on the error for any values of
$h$. For the Neumann problem it's even worse because, in addition to
the slow decay rate, we have a discontinuity which is detectable to
machine precision for all values of $L$ considered. Hence, 
decreasing $h$ will not decrease the error because the discontinuity 
will not vanish.

\begin{figure}
\centering
\subfloat[]{\includegraphics[width=.5\textwidth]{ExtDirchError.jpg}} \\
\subfloat[]{\includegraphics[width=.5\textwidth]{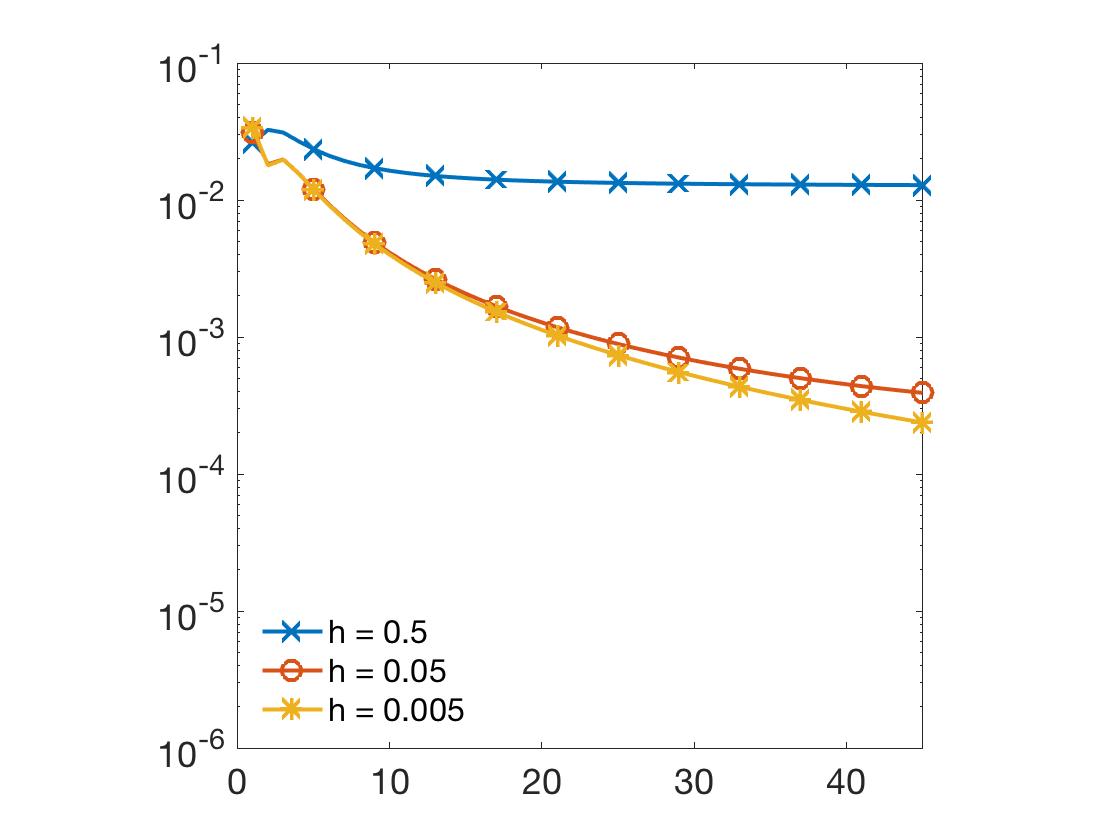}}
\subfloat[]{\includegraphics[width=.5\textwidth]{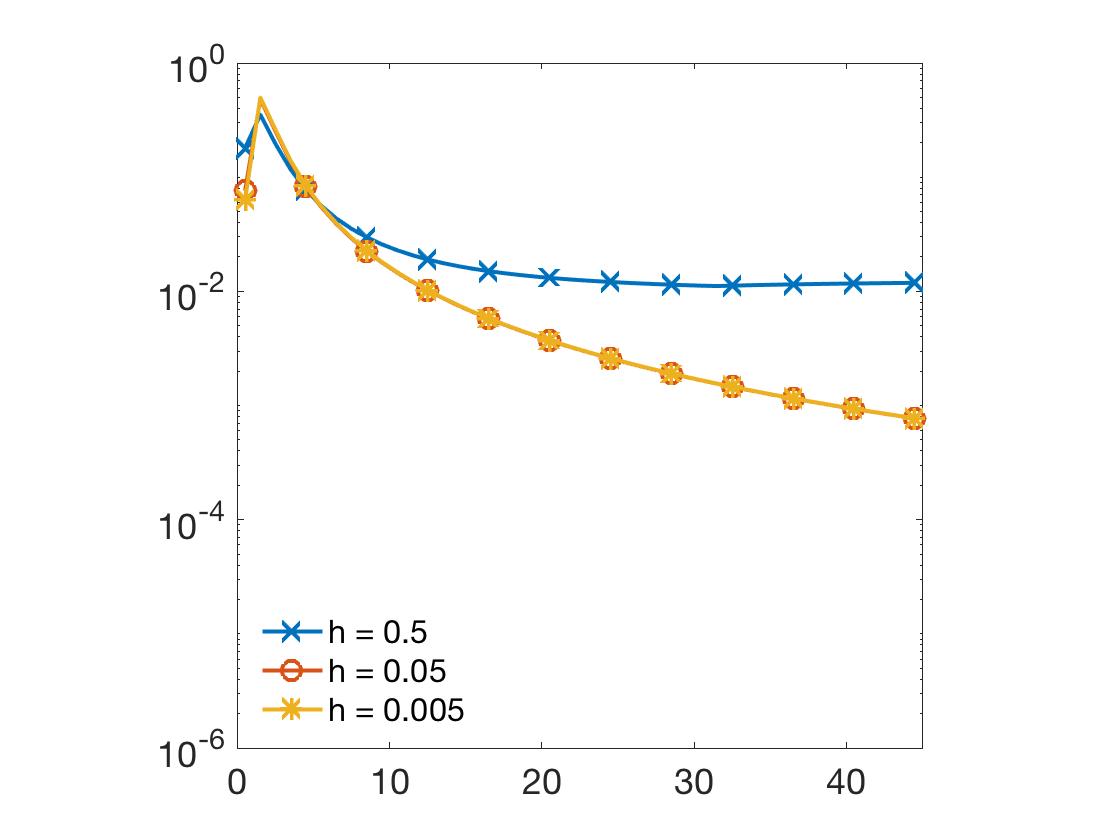}}
\caption{$L^\infty$ error between the numerical solution and the true
 solution for varying values of $L$; each curve represents a
 different choice in the spatial step size $h$. Panel (a) shows the
 results when the algebraic decay is taken into account. Panel (b)
 corresponds to homogeneous Dirichlet boundary conditions. Panel (c)
 corresponds to homogeneous Neumann boundary conditions where we've
 taken $\tilde{L} = \frac{L}{2}$.}\label{Fig: Fig5}
\end{figure}

\subsection{Dirichlet Problem with Not Strictly Positive Kernel}
\label{s:Dirichlet_kernel_nonposi}
Consider the Dirichlet problem
\begin{equation}
\begin{cases}
\mathcal{L} \ast u(x) = f(x), & x \in (-L, L), \\ 
u(x) = \text{sech}(x), & x \in (-L,L)^c,
\end{cases}
\end{equation}
with
\begin{align*}
\nu(y) &= \frac{3}{2}e^{-|y|}- 2e^{-2|y|}, \\[.1cm]
 f(x) &=  4\,{{e}}^{-x}+\frac{1}{{\cosh}\left(x\right)}-\frac{3\,\ln\left({{e}}^{2\,x}+1\right)\,{{e}}^{-x}}{2}+2\,{{e}}^x\,\left(2\,{\tan^{-1}}\left({{e}}^x\right)\,{{e}}^x-\pi \,{{e}}^x+2\right)\\
 &-4\,{\tan^{-1}}\left({{e}}^x\right)\,{{e}}^{-2\,x}+\frac{3\,{{e}}^x\,\left(2\,x-\ln\left({{e}}^{2\,x}+1\right)\right)}{2}
.
\end{align*}
Then by direct calculation, we have that
\begin{gather*}
f_1(h) =  \frac{1}{h^2}\big[\frac{{{e}}^{-2\,h}\,\left(2\,h^2+2\,h+1\right)}{2}-\frac{3\,{{e}}^{-h}\,\left(h^2+2\,h+2\right)}{2}+\frac{5}{2}\big] = -\frac{1}{6}h + O(h^2) ,\\[.2cm]
f_2(h) = -\frac{1}{6} h^5 + O(h^6) \quad , \quad f_3(h) = -\frac{1}{10} h^5 + O(h^6) \quad , \quad f_4(h) = \frac{1}{2} h^2 + O(h^3).
\end{gather*}
Hence, we should generically expect the scheme to converge at rate
$O(h^2)$.

In this particular case, $\nu(y)$ has the antiderivatives
\begin{equation*}
F(y) = \frac{3}{2}{e}^{-\left|y\right|}-\frac{1}{2}{{e}}^{-2\,\left|y\right|} \; \;  , \; \;  F'(y) = -{\text{sign}}\left(y\right)\,\left(\frac{3}{2}\,{{e}}^{-\left|y\right|}-{{e}}^{-2\,\left|y\right|}\right) \; \;  , \; \;  F''(y) = \frac{3}{2}e^{-|y|}- 2e^{-2|y|},
\end{equation*}
so that now  all of the $w_j$ can  be calculated. We also have that the integral $A$ can be directly
computed and is given by
\[A = 3 e^{-2L_W} - 2 e^{-4L_W}.\]
The integral $B_i$ can also be calculated directly and is given by
\begin{align*}
B_i =& 4\,{{e}}^{-2\,x_i}\,{\tan^{-1}}\left({{e}}^{-L}\,{{e}}^{x_i}\right)-4\,{{e}}^{-L-x_i}+\frac{3\,{{e}}^{-x_i}\,\left(2\,x_i+\ln\left({{e}}^{-2\,L}+{{e}}^{-2\,x_i}\right)\right)}{2}\\[.1cm]
&-\frac{3\,{{e}}^{x_i}\,\left(2\,x_i-\ln\left({{e}}^{-2\,L}+{{e}}^{2\,x_i}\right)\right)}{2}-4\,{{e}}^{x_i}\,\left({{e}}^{-L}-{\tan^{-1}}\left({{e}}^{-L}\,{{e}}^{-x_i}\right)\,{{e}}^{x_i}\right)
.
\end{align*}

We then have all the necessary quantities to implement the numerical
scheme for the Dirichlet problem. First note that the true solution of
this problem is in fact $u(x) = \text{sech}(x)$; this can be confirmed
by a straightforward integration. Fig.~\ref{Fig: Fig6}(a) shows a plot
of the solution $u(x)$ and the forcing function $f(x)$. 
Fig.~\ref{Fig: Fig6}(b) shows the $L^\infty$ error between the numerical solution
and the true solution for varying values of $L$ and $h$. We see that
for fixed $L$ the scheme does indeed converge at an $O(h^2)$ rate.

\begin{figure}
\subfloat[]{\includegraphics[width=.5\textwidth]{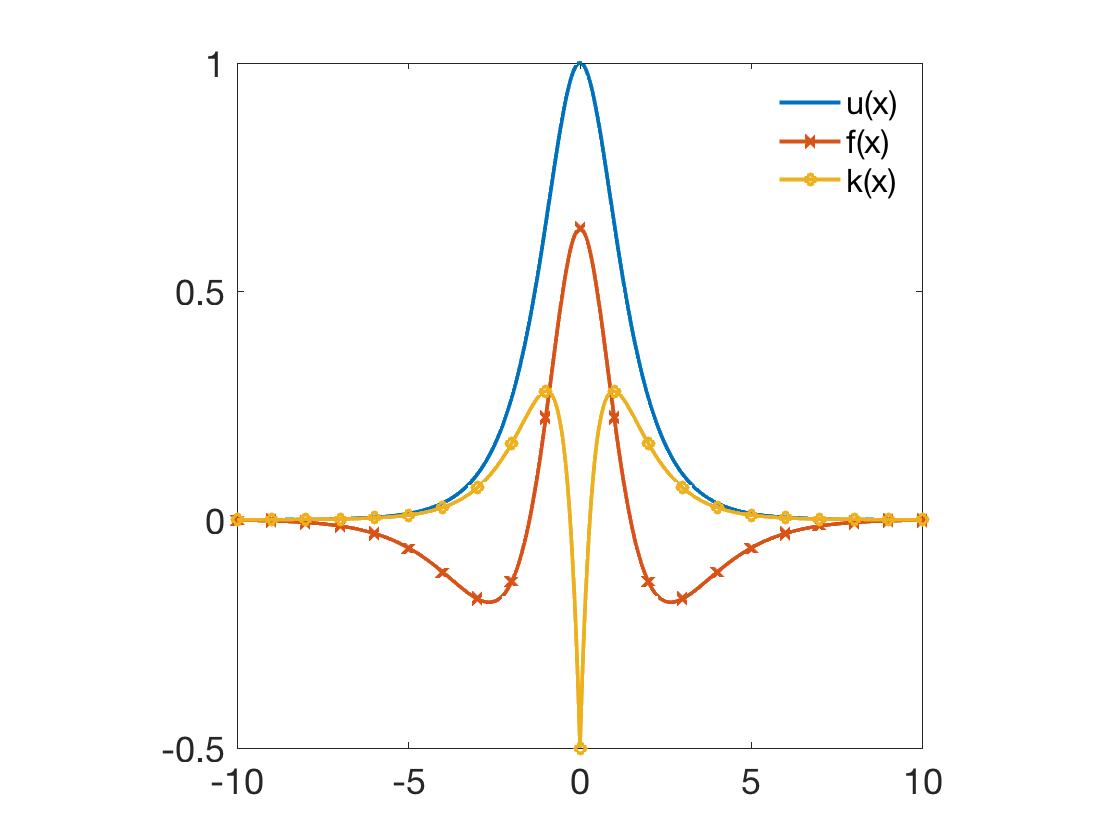}}
\subfloat[]{\includegraphics[width=.5\textwidth]{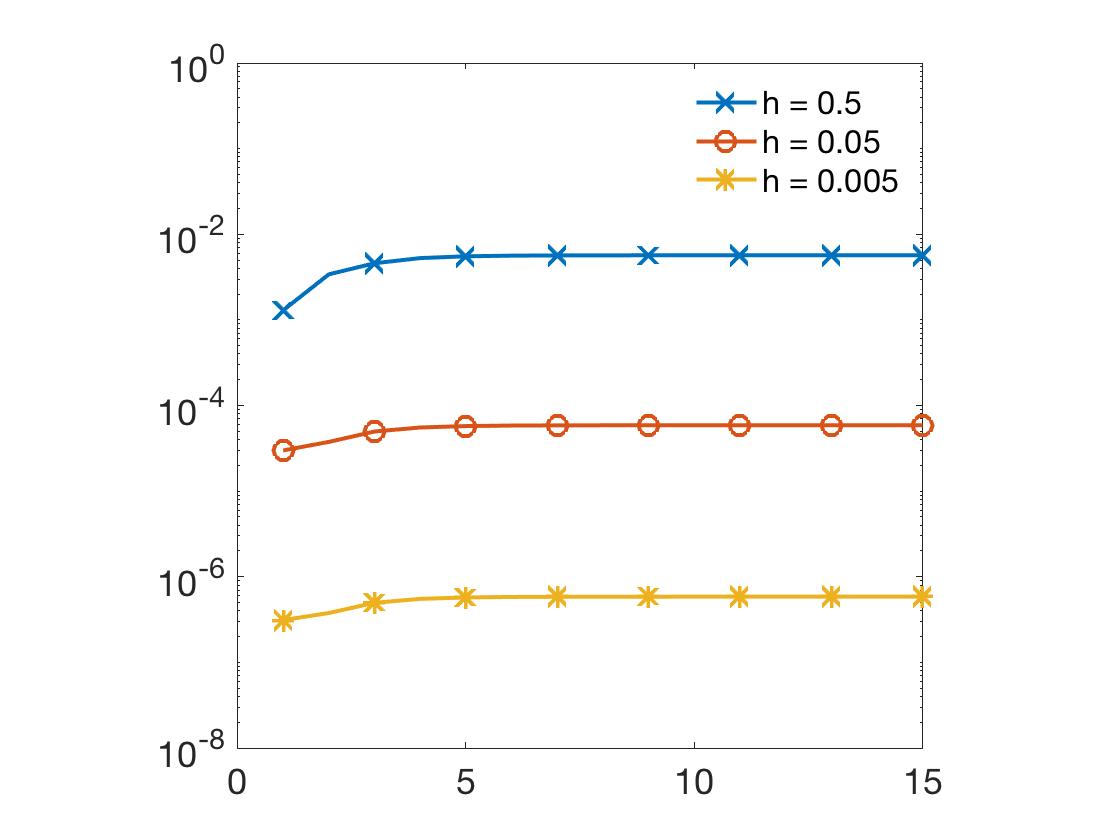}}
\caption{(a) Plot of the solution $u(x) = \text{sech}(x)$ and the
corresponding forcing function $f(x)$. (b) Plot of he $L^\infty$
error between the numerical solution and the true solution for
varying values of $L$; each curve represents a different choice in
the spatial step size $h$.}\label{Fig: Fig6}
\end{figure}

\section{Conclusion}

In this paper we consider integro-differential equations that model 
the evolution process of quantities that experience a nonlocal form of dispersion.
In particular, we look at diffusion processes that are modeled
using convolution kernels that do not have compact support and 
decay exponentially at infinity. We develop algorithms for finding the steady states
of these systems. 

In contrast to previous methods which assume local boundary data, 
our numerical method accounts for the correct nonlocal nature of the boundary conditions.
We present three numerical schemes addressing the case of nonlocal Dirichlet
boundary conditions (DP), Neumann boundary conditions (NP), 
and the whole real line problem (RP).

When the equation is posed on the whole real line, we show that a
unique solution exist, provided the right hand side decays 
at least algebraically and has zero mean and first moment. 
More importantly, the result shows that there is a relation between the decay of 
the right hand side and the decay of the solution. 
This information is then used to approximate the solution 
outside the computational domain and thus develop a scheme for RP.

 Since nonlocal Neumann boundary conditions require us to approximate the solution 
outside a bounded domain, the numerical schemes for NP and RP are 
almost identical. In both cases the scheme boils down to inverting a matrix equation. 
We are able to show using a Neumann series that this matrix is invertible
provided the convolution kernel is nonnegative. 
This also proves the convergence of the scheme.

For the Dirichlet problem, we consider kernels that can take on negative values,
but that have positive tails. We show, using the theory of Toeplitz matrices,
 that this is enough to prove the convergence of the
our numerical scheme. This is an improvement over previous results which are based on 
maximum principles and therefore require nonnegative kernels with compact support.

Finally, for applications where the model equations are posed on the whole real line, 
there is always a question of what are the best boundary conditions one can use
to approximate solutions.
Here we present a numerical scheme that does not require explicit
boundary conditions. However, we do find that in certain circumstances,
mainly when the right hand side decays exponentially,  the
Dirichlet problem provides a more efficient method for 
approximating the whole real line problem. First, because one 
can reduce the size of the computational domain, and secondly one does not 
have to approximate the solution outside this domain,
i.e. setting the solution to zero gives a good approximation.

\appendix

\section{ Nonlocal Flux, Gauss Theorem, and an Example}\label{s:calculus}
In this section we summarize results presented in \cite{gunzburger2012, gunzburger2013}, 
which generalize the concept of flux to include short as well as long
range movement of particles. 
Then, in Subsection \ref{ss:example} we use a very general and well known population model 
as an example of how this generalized version of flux, together with
conservation of mass, gives rise to equation \eqref{e:evolution}.
Similar derivations have been done in \cite{andreu2010}.

As already pointed out in the introduction
variations of equation \eqref{e:evolution} 
have been introduced in other contexts. 
Here we restrict ourselves to the population model, since we believe it provides
a simple  example where one can apply the notion of nonlocal Neumann boundary conditions
presented in \cite{gunzburger2012, gunzburger2013}. 
For more information about other nonlocal models, the review paper by
Fife \cite{fife2003} provides a good starting point for the case.

 \subsection{Nonlocal Flux}
To give an intuitive notion of what constitutes a nonlocal flux, we
first recall the traditional definition of this term. In physical
applications flux represents the rate of motion per unit area of a
quantity $u$ (fluid, concentration, number of particles) across some
boundary. Implicit in this definition is the assumption that the
transport of this quantity happens at small scales. However, in
certain applications transport can occur over long, as well as short,
spatial scales. Consider for example an area of vegetation with seeds
that can travel close to as well as far from their originating
organisms thanks to wind currents. In this case flux is no longer
proportional to a local quantity, like $u$ (transport equation) or the
gradient of $u$ (diffusion equation), but instead should be expressed
through a nonlocal operator.

We can make these ideas more precise by looking again at our
vegetation example. For simplicity assume for now that we only have
one organism at position $y$, whose seeds are entering a field
$\Omega$. Suppose as well that we have a function $\phi(x,y,t)$ that
tells us the proportion of seeds from position $y \notin \Omega$ that
fall in location $x \in \Omega$ per unit time. Then the flow of seeds
from $y$ into region $\Omega$ is given by the integral
\[ \int_\Omega \phi(x,y,t)\;dx.\]

More generally, one can construct a function $\psi(x,y,t)$ such that
\[ \int_\Omega \psi(x,y,t)\;dx,\]
 represents a nonlocal flux density. Then, the expression
\[ \int_{\Omega_1} \int_{\Omega_2} \psi(x,y,t) \;dx\;dy\]
gives us the net nonlocal flux from region $\Omega_1$ into region
$\Omega_2$. If this expression is positive then indeed we have net
flux from $\Omega_1$ into $\Omega_2$. On the other hand, if this
quantity is negative, then the net flow occurs in the reverse
direction.

For the above definition to be consistent with our intuition of how
flux should behave, one imposes an \emph{action-reaction
  principle}. Given two distinct domains $\Omega_1$, $\Omega_2$ we
would like for the nonlocal flux from $\Omega_1 $ into $\Omega_2$ to
be equal in magnitude, but of opposite sign, as the the nonlocal flux
from $\Omega_2$ into $\Omega_1$, i.e.
\[ \int_{\Omega_1} \int_{\Omega_2} \psi(x,y,t) \;dx \;dy + \int_{\Omega_2} \int_{\Omega_1} \psi(y,x,t) \;dy \;dx =0.\]
It is straightforward to check that this holds provided $\psi$ is
antisymmetric in $x$ and $y$, that is $\psi(x,y,t) = -\psi(y,x,t)$. 
Notice that this condition also implies that there are
\emph{no self interactions}, meaning that
\[\int_\Omega \int_\Omega \psi(x,y,t) \;dx \;dy=0.\]

 \subsection{Nonlocal Gauss' Theorem}
Given a bounded domain $\Omega$, Gauss' Theorem relates the total flux
across the boundary $\partial \Omega$, in terms of a volume integral
over the domain, $\Omega$. More precisely, if ${\bf F}$ represents a
smooth vector field and ${\bf n}$ the unit normal to $\Omega$, then
 \[  \int_\Omega \nabla \cdot {\bf F} \;dV = \int_{\partial \Omega} {\bf F \cdot n } \;dS. \]
 
 In the nonlocal case, the action-reaction principle provides an
 analogue to Gauss' Theorem since it relates the flux from $\Omega^c$
 into $\Omega$ in terms of an integral over $\Omega$,
 \[  \int_{\Omega^c} \int_{\Omega} \psi(x,y,t) \;dx \;dy =  \int_{\Omega} \int_{\Omega^c} - \psi(x,y,t) \;dx \;dy . \]

\subsection{Derivation equation with population dynamics example}\label{ss:example}
To illustrate our point suppose we are interested in the evolution of
a population that is able to move short as well as long distances. Let
$u(x,t)$ denote the density of this population at time $t$ and
location $x$, and let $\Omega \subset \R^n$ be a bounded domain.  Then
$m = \int_\Omega u(x,t) \;dx$ represents the total number of
individuals in $\Omega$ at time $t$.

Assume as well that the fraction of the population that flows from $x$
to $y$ depends on the distance between these two points and is
proportional to the density at point $x$. More precisely, at time $t$
the flow from $x\in \Omega$ to $y \in \Omega^c$ is is given by the
product $u(x,t)\nu(x,y)$, where we also assume that the kernel $\nu(x,y) =\nu(|x-y|)$
is symmetric and exponentially decaying. Then, the flow
from the point $x \in \Omega$ to the domain $ \Omega^c$ is given by
the integral 
\[\int_{\Omega^c} u(x,t) \nu(|x-y|) \;dy,\]
and the total flow out of $\Omega$ can be represented by the
expression
\[ \int_{\Omega} \int_{\Omega^c} u(x,t) \nu(|x-y|) \;dy\;dx.\]
Similarly, the flow from $ \Omega^c$ into the point $x \in \Omega$ can
be written as
\[ \int_{\Omega^c} u(y,t) \nu(|x-y|) \;dy\]
so that the total flow into $\Omega$ is
\[ \int_{\Omega} \int_{\Omega^c} u(y,t) \nu(|x-y|) \;dy\;dx.\]

Combining these two expressions we find that the net flow, $Q$, of the
population in/out of the domain $\Omega$ is given by
\begin{align*}
Q = & - \int_{\Omega} \int_{\Omega^c} u(x,t) \nu(|x-y|) \;dy\;dx
  +\int_{\Omega} \int_{\Omega^c} u(y,t) \nu(|x-y|) \;dy\;dx
\\ 
Q = & \int_{\Omega} \int_{\Omega^c} - (u(x,t) -u(y,t)) \nu(|x-y|)\;dy\;dx.
\\ 
Q = & \int_{\Omega} \int_{\R^n} \underbrace{ - (u(x,t) -u(y,t)) \nu(|x-y|)}_{=\psi(x,y,t)} \;dy\;dx.
\end{align*}
where the last line follows from the flow density function
$\psi(x,y,t)$ being antisymmetric, and the fact that this rules out
self-interactions.
 
By conservation of mass
\[ m_t =  \int_\Omega u_t \;dx  =\int_{\Omega}  \int_{\R^n}  - (u(x,t) -u(y,t)) \nu(|x-y|) \;dy\;dx + \int_\Omega f\;dx,\]
where $f(x)$ is a density specifying the net loss/gain of individuals
at location $x$ that combines births and deaths. Since the domain
$\Omega$ is arbitrary, the result is an evolution equation for the
variable $u$,
 \begin{equation}
u_t = -\mathcal{L} \ast u + f(x) \quad \mbox{for} \quad x \in \Omega
\end{equation}
where
\[\mathcal{L} \ast u  =  \int_{\R^n }   (u(x,t) -u(y,t))  \nu(|x-y|) \;dy. \]

Now, because the flow is nonlocal, instead of boundary conditions
across $\partial \Omega$, we need to impose conditions on $\Omega^c$:
\begin{itemize}
\item Dirichlet: One specifies the value of the function $u$ in $\Omega^c$
\[u(x)= g(x) \quad \mbox{for} \quad x \in \Omega^c.\]
\item Neumann: Since by the nonlocal Gauss' Theorem 
  $\int_{\Omega^c} \int_{\R^n}- \psi(x,y,t) \;dy\;dx $ represents the net flow in/out
  of $\Omega^c$, then the following relation specifies a given and
  fixed flow density, $f_c(x)$, from $\Omega^c$ into the domain $\Omega$,
\[  \int_{\R^n} \psi(x,y,t) \;dy  = f_c(x) \quad x \in \Omega^c.\]
\item Mixed: Given $\Omega^c = \Omega^c_1 \cup \Omega^c_2$
\begin{align*}
u(x) = g(x) & \quad \mbox{for} \quad x\in \Omega^c_1,
\\
\int_{\R^n} \psi(x,y,t) \;dy = f_c(x ) & \quad \mbox{for} \quad x \in \Omega^c_2.
\end{align*}
\end{itemize} 

\section{Nonlocal Diffusive Operators on a Lattice}
\label{a:well_posedness_nonlocal_lattice}

 Here we extend the results of  \ref{s:weighted_spaces} to operators defined
on lattices using results presented in \cite{jaramillo2019}.

We first define the analogue of the spaces $L^2_\gamma(\R), M^{2,s}_\gamma(\R),$ 
and $H^s_\gamma(\R)$ in the obvious way and
denote them by $\ell_\gamma^2(\Z)$, $m_\gamma^{2,s}(\Z)$,
$h^s_\gamma(\Z)$, respectively. As above we use $\langle, \rangle$ to
denote the pairing between dual elements, and we let $\underline{u} = \{u_j\}_{j\in \Z}$.
 In the case of lattices, the Fourier Transform is given by
\[
\begin{array}{c c c c}
\mathcal{F}_d : & \ell^2(\Z) & \longrightarrow & L^2(\mathbb{T}_1)\\ &
u = \{ u_j \}_{j \in \Z} & \longmapsto & \hat{u}(\sigma) = \sum_{j\in \Z} u_j \rme^{-2 \pi \rmi j \sigma}
\end{array}
\]
where $\mathbb{T}_1 = \R/ \Z$ is the unit circle. We can also define
discrete derivatives for elements in $\ell^2(\Z),$
\[ \delta_+ (\{ u_j \}_{j \in \Z} ) = \{ u_{j+1} - u_j \}_{j \in \Z}
 \qquad \delta_-(\{ u_j \}_{j \in \Z} ) = \{ u_{j} - u_{j-1} \}_{j \in \Z} 
 \qquad  \delta  = - \rmi ( \delta_+ + \delta_-)/2 \] 
with their corresponding Fourier symbols,
\[ D_+ (\sigma) = \rme^{2 \pi \rmi \sigma} -1, 
\quad D_-(\sigma) = 1 - \rme^{-2 \pi \rmi \sigma} ,
\quad D(\sigma) = \sin (2\pi \sigma).  \]
As was the case for operators defined on $L^2(\R)$, the Fourier
Transform of general convolution operators, $ L(\sigma) $, is a
multiplication symbol defined on the space $L^2(\mathbb{T}_1)$:
\[
\begin{array}{c c c c}
\hat{ \mathcal{L}} : & D( \hat{ \mathcal{L}} ) \subset L^2(\mathbb{T}_1) & \longrightarrow & L^2(\mathbb{T}_1)\\ 
& u(\sigma) & \longmapsto & L(\sigma) u(\sigma)
\end{array}
\]

Here, we make the following assumptions on the Fourier symbol
$L(\sigma)$, of our discrete convolution operators $\mathcal{L}$.
\begin{Hypothesis}\label{h:hypo_lattice}
The symbol $L(\sigma)$ is analytic, uniformly bounded, and 1-periodic
in a strip $\Omega_1 = \R \times (\rmi \sigma_1, \rmi \sigma_1)$ for
some $\sigma_1>0$. Moreover, when restricted to $\sigma \in [-1/2, 1/2]$
the symbol $L(\sigma)$ is invertible except at $\sigma =0$,
where it has a zero of multiplicity $m$.
\end{Hypothesis}

As was the case in the real line, the operator $\mathcal{L}$ can be
decomposed into an invertible operator $\mathcal{M}_L(\sigma)$ and a
Fredholm operator with Fourier symbol 
$ (\rme^{2 \pi \rmi \sigma} -1)^m [ 1 + \rmi C \sin^l(2 \pi \sigma)]^{-1} $, 
which formally corresponds to $\delta_+^m ( 1+ \delta^l)^{-1}$. 
This leads to the following proposition.

\begin{Proposition}\label{p:discrete_fredholm}
For $\gamma \notin \{ 1/2, 3/2, \cdots, m-1/2\}$, and appropriate
value of the integer $l$, the operator $ \mathcal{L} :
m^{2,m}_\gamma(\Z) \longrightarrow h^l_{\gamma+m}(\Z) $ satisfying
Hypothesis \ref{h:hypo_lattice} is Fredholm. Moreover, letting
$\underline{\eta}^b = \{ \eta^b\}_{\eta \in \Z}$, we have that
\begin{itemize}
\item for $\gamma > - 1/2$ the operator is injective with cokernel
\[ \coker = \mathrm{span} \{ \underline{\eta}^\beta \mid \beta = 0,1, \cdots, m-1\}\]
\item for $\gamma < 1/2-m $ the operator is surjective with kernel
\[ \ker = \mathrm{span} \{ \underline{\eta}^\beta \mid \beta = 0,1, \cdots, m-1\}\]
\item for $j - 1/2 -m < \gamma< j- m +1/2 $, with $ j \in \Z $, $1 \leq j <m$ , the operator has kernel
\[ \ker = \mathrm{span} \{ \underline{\eta}^\beta \mid \beta = 0,1, \cdots, m-j-1\}\]
and cokernel
\[ \coker = \mathrm{span} \{ \underline{\eta}^\beta \mid \beta = 0,1, \cdots, j-1\}.\]
\end{itemize}
On the other hand, the operator does not have closed range when
$\gamma \in \{ 1/2, 3/2, \cdots, m-1/2\}$.
\end{Proposition}

From the above result we obtain the next corollary, which we will use
in the proof of convergence for the numerical schemes. In particular,
the following result gives us information about the decay rate of
solutions to the discrete convolution problem defined over
$\ell^2(\Z)$
\[ \mathcal{L} \ast \underline{u} = \underline{ f}. \]

\begin{Corollary}\label{c:uniqueness_d}
Consider the discrete convolution operator $\mathcal{L}$, with Fourier
symbol $(\rme^{2 \pi \rmi \sigma} -1)^2 [1+ (2\rmi \sin (2\pi \sigma)^2]^{-1}$, 
and defined as
\[
\begin{array}{c c c}
 \mathscr{L}: m^{2,2}_{\gamma}(\Z) & \longrightarrow & h^{2}_{\gamma+2}(\Z)\\ 
 \underline{u} &\longmapsto & \mathcal{L} \ast \underline{u}.
 \end{array}
\]
Suppose $\gamma > -1/2 $, then the equation $\mathcal{L} \ast \underline{u} =\underline{f}$
has a unique solution, with $|u_j| < C |j| ^{1/2-(\gamma+1)}$ for large $|j|$, 
provided the right hand side $\underline{f} \in h^2(\Z)_{\gamma+2}$ satisfies
\[ \langle \underline{f}, \underline{ 1} \rangle = 0 
\qquad \mbox{and} 
\qquad \langle \underline{f}, \underline{\eta} \rangle = 0.\]
\end{Corollary}

\section{Calculation of $w_j$}\label{a:weights_calc}
Here we derive formulas for the weights $w_j$ considered
earlier. Assume that $\nu(y)$ has two integrable
antiderivatives. Namely, assume there exists a function $F$ such that
$F''(y) = \nu(y)$. We can then use integration by parts to show
the following results.

In the first case, assume $1 < |j| < M$, then
\begin{align*}
\int_{h \leq |y| \leq L_W} T(y-x_j) \nu(y) \; dy &= \int^{x_{j+1}}_{x_{j-1}} T(y-x_j) \nu(y) \; dy\\[.2cm]
&= \int^{h}_{-h} T(z) \nu(z+x_j) \; dz\\[.2cm]
&= \frac{1}{h}\bigg[F(x_{j+1}) -2F(x_{j}) + F(x_{j-1})\bigg].
\end{align*}
The second line is obtained by the change of variables $z = y - x_j$. The
third line is obtained by doing integration by parts twice and using $\nu = F^{''}$.\\

When $j=1$, we don't integrate over the full hat function because half
of the hat function is not in the domain of $|y|\geq h$. It reads:
\begin{align*}
\int_{h \leq |y| \leq L_W} T(y-x_1) \nu(y) \; dy &= \int^{x_{2}}_{x_{1}} T(y-x_1) \nu(y) \; dy\\[.2cm]
&= \int^{h}_{0} T(z) \nu(z+x_1) \; dz\\[.2cm]
&= \-F'(x_1) + \frac{1}{h}\bigg[F(x_{2}) -F(x_{1})\bigg].
\end{align*}

When $j = M$ a similar argument shows
\[\int_{h \leq |y| \leq L_W} T(y-x_M) \nu(y) \; dy =F'(x_M) + \frac{1}{h}[F(x_{M-1}) -F(x_{M})].\]

Finally, because both $\nu(y)$ and $T(y)$ are even we have
\begin{align*}
\int_{h \leq |y| \leq L_W} T(y-x_{-j})\nu(y) \; dy &= \int_{h \leq |y|  \leq L_W} T(y+x_{j})\nu(y) \; dy \\[.2cm]
&= \int_{h \leq |z| \leq L_W} T(-z+x_{j})\nu(-z) \; dz\\[.2cm]
&= \int_{h \leq |z| \leq L_W} T(z-x_{j})\nu(z) \; dz .
\end{align*}
Looking at the definitions in the previous sections, this immediately
gives
\begin{equation*}
w_{-j} = w_j.
\end{equation*} 

To summarize, we've shown that if the kernel $\nu$ has the appropriate
antiderivatives $F$ then the weights $w_j$ are given explicitly by
\begin{equation*}
w_j =
\begin{cases}
f_1(h) - F'(x_1) + \frac{1}{h}[F(x_{2}) -F(x_{1})], &|j| =  1 \\[.2cm]
\frac{1}{h}[F(x_{j+1}) - 2F(x_{j})+F(x_{j-1})], & 1< |j| <M \\[.2cm]
F'(x_M) + \frac{1}{h}[F(x_{M-1}) -F(x_{M})], & |j| = M
\end{cases}
\end{equation*}
with $w_0=0$.

\bibliographystyle{siamplain}
\bibliography{nonlocalBIB}

\end{document}